\documentclass[11pt]{article}
\usepackage{amssymb,amsmath,amsthm}
\usepackage[all]{xy}
\usepackage{hyperref}

\newtheorem*{thrm}{Theorem}
\newtheorem*{lemma}{Lemma}

\newtheorem{thrm*}{Theorem}
\newtheorem{lemma*}[thrm*]{Lemma}
\newtheorem{prop}[thrm*]{Proposition}
\newtheorem{defn}[thrm*]{Definition}

\newcommand{\mcg}{\operatorname{\ensuremath{MCG}}}
\newcommand{\stab}{\operatorname{Stab}}
\newcommand{\sym}{\operatorname{Sym}}
\newcommand{\LL}{(L_1,\dots,L_n)}

\newcommand{\caA}{\mathcal{A}}
\newcommand{\caB}{\mathcal{B}}
\newcommand{\caC}{\mathcal{C}}
\newcommand{\caD}{\mathcal{D}}
\newcommand{\caH}{\mathcal{H}}
\newcommand{\caP}{\mathcal{P}}
\newcommand{\caM}{\mathcal{M}}
\newcommand{\caR}{\mathcal{R}}
\newcommand{\caS}{\mathcal{S}}
\newcommand{\caT}{\mathcal{T}}

\newcommand{\Tg}{\mathcal{T}_g}
\newcommand{\Tgn}{\mathcal{T}_{g,n}}
\newcommand{\TgL}{\mathcal{T}_g(L)}

\newcommand{\Mg}{\mathcal{M}_g}
\newcommand{\Mgn}{\mathcal{M}_{g,n}}

\newcommand{\Tbar}{\overline{\mathcal{T}}}
\newcommand{\Tbarg}{\overline{\mathcal{T}}_g}
\newcommand{\Tbargn}{\overline{\mathcal{T}}_{g,n}}

\newcommand{\That}{\widehat{\mathcal{T}}_{g,n}}
\newcommand{\Mbar}{\overline{\mathcal{M}}}

\newcommand{\Mbargn}{\overline{\mathcal{M}}_{g,n}}
\newcommand{\MbargL}{\overline{\mathcal{M}}_g(L)}
\newcommand{\Cbar}{\overline{\mathcal{C}}}
\newcommand{\Cg}{\mathcal{C}_g}
\newcommand{\Cgn}{\mathcal{C}_{g,n}}
\newcommand{\omTL}{\omega_{\TgL}}
\newcommand{\omThat}{\omega_{\That}}

\newcommand{\lla}{\ell_{\alpha}}

\usepackage{graphicx}
\graphicspath{%
    {converted_graphics/}
    {CBMS/CBMSbook/}
}
\begin{document}
 
\title{Lectures and notes: Mirzakhani's volume recursion and approach for the Witten-Kontsevich theorem on moduli tautological intersection numbers}         
\author{Scott A. Wolpert\footnote{Partially supported by National Science Foundation grant DMS - 1005852.}}        
\date{\today}          
\maketitle
\vspace{-.1in}

\begin{figure}[htbp] 
  \centering
  \includegraphics[bb=0 0 509 432,width=3.5in,height=2.97in,keepaspectratio]{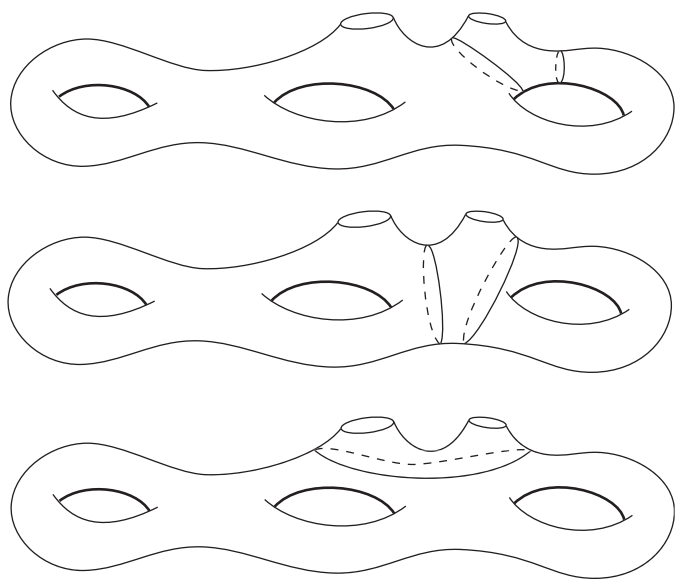}
  \caption{The boundary pants configurations for the length identity.}
  \label{fig:recursion}
\end{figure}

The following materials were presented in a short course at the 2011 Park City Mathematics Institute, Graduate Summer School on Moduli Spaces of Riemann Surfaces.  Brad Safnuk assisted in the preparation for and running of the course.  It is my pleasure to thank Brad for his assistance.  

\newpage
\section{Introduction.}
\noindent {\bfseries The papers.}

\noindent Simple geodesics and Weil-Petersson volumes of moduli spaces of bordered Riemann surfaces. Invent. Math., 167(1):179-222, 2007.
\vspace{.01in}

\noindent Weil-Petersson volumes and intersection theory on the moduli space of curves. J. Amer. Math. Soc., 20(1):1-23, 2007. 
\vspace{.1in}

\noindent {\bfseries The goals of the papers.} 

\noindent 1. Derive an explicit recursion for  WP moduli space volume polynomials.

\noindent 2. Apply symplectic reduction to show that the polynomial coefficients are intersection numbers.

\noindent 3. Show that the volume recursion satisfies the Virasoro relations - Witten's conjecture.

\vspace{.1in} 

The trio of Maryam Mirzakhani papers \cite{Mirvol,Mirwitt, Mirgrow} are distinguished for involving a large number of highly developed considerations.  The first work requires a detailed description of Teichm\"{u}ller space, the action of the mapping class group, formulas for Weil-Petersson (WP) symplectic geometry, classification of simple geodesic arcs on a pair of pants, the length infinite sum identity and exact calculations of integrals.  The second work involves a description of the moduli space tautological classes $\kappa_1$ and $\psi$, as characteristic classes for $S^1$ principal bundles in hyperbolic geometry, the moment map and exact symplectic reduction, as well as combinatorial calculations.  The third work uses the $PL$ structure of Thurston's space of measured geodesic laminations $\mathcal{MGL}$, the train-track symplectic form and Masur's result that the mapping class group acts ergodically on $\mathcal{MGL}$.  A fine feature of the works is that they showcase important aspects of the geometry, topology and deformation theory of Riemann/hyperbolic surfaces.  Mirzakhani's recursion for volume is applied in all three works and in a current preprint.   A discussion placing Mirzakhani's work into the context of approaches to Witten-Kontsevich theory is given in the first section of the Mulase-Safnuk paper \cite{SafMu}.  A brief exposition of Kontsevich's original solution \cite{Kont} of Witten's conjecture, including the basic geometry of tautological classes on moduli space and the ribbon graph expansion of matrix integrals is given in the Bourbaki Seminar of Looijenga \cite{Looij}.  A brief exposition of Mirzakhani's volume recursion, solution of Witten-Kontsevich and applications of WP volume limits are given in \cite{Dosurv}.  An overall exposition of Mirzakhani's {\em prime simple geodesic theorem} \cite{Mirgrow} is given in \cite[Chaps. 9, 10]{Wlcbms}.  Mirzakhani's work is just one part of a subject with high activity and many active researchers; Google Scholar shows 200 citations to the three Mirzakhani works, and much more generally over 1100 citations to Witten's original papers \cite{Witt1,Witt2} on two-dimensional gravity and gauge theories.  

\vspace{.1in}

\noindent First paper brief. A symplectic fibered product decomposition for covers of the moduli space of bordered Riemann/hyperbolic surfaces is combined with the $d\ell\wedge d\tau$ formula for the symplectic form and a universal identity for sums of geodesic-lengths to derive an explicit recursion for computing volume.  The volume of the moduli space of genus $g$, $n$ boundaries surfaces  is shown to be a polynomial with positive coefficients in surface boundary lengths of total degree $6g-6+2n$.
\vspace{.1in}

\noindent Second paper brief.  Twist Hamiltonian flows are combined with the boundary length moment map to apply symplectic reduction for the family of moduli spaces of bordered hyperbolic surfaces.   The consequence is that the coefficients of the volume polynomials are moduli space characteristic numbers.  A geometric construction shows that the characteristic numbers are tautological intersection numbers. The volume recursion is shown to satisfy Virasoro algebra constraints. 

\vspace{.1in}
The following theorems are the immediate take away results of the papers. 

\begin{thrm} \textup{The WP volume polynomials.}  The volume polynomials are determined recursively from the volume polynomials of smaller total degree, \cite[Formula (5.1) \& Theorem 8.1]{Mirvol}.   The volume $V_{g,n}(L_1,\dots,L_n)$ of the moduli space of genus $g$, $n$ boundaries, hyperbolic surfaces with boundary lengths $L=(L_1,\dots,L_n)$ is a polynomial
\[
V_{g,n}(L)\,=\,\sum_{\stackrel{\alpha}{|\alpha|\le 3g-3+n}} C_{\alpha}\,L^{2\alpha},
\]
for multi index $\alpha=(\alpha_1,\dots,\alpha_n)$ and where $C_{\alpha}>0$ lies in $\pi^{6g-6+2n-2|\alpha|}\mathbb Q$, \cite[Theorems 1.1 \& 6.1]{Mirvol}.  The coefficients are intersection numbers given as 
\[
C_{\alpha}=\frac{2^{\delta_{1g}\delta_{1n}}}{2^{|\alpha|}\alpha!(3g-3+n-|\alpha|)!}
\,\int_{\overline{\caM}_{g,n}}\psi_1^{\alpha_1}\cdots\psi_n^{\alpha_n}\omega^{3g-3+n-|\alpha|},
\]
where $\psi_j$ is the Chern class for the cotangent line along the $j^{th}$ puncture, $\omega$ is the symplectic form, $\alpha!=\prod_{j=1}^n\alpha_j!$, and $\delta_{**}$ is the Kronecker indicator delta, \cite[Theorem 4.4]{Mirwitt}. 
\end{thrm}
\begin{thrm} \cite[Theorems 6.3 \& 6.4]{Mirvol}.  \textup{Recursive relations for the volume polynomial leading coefficients.} For a multi index $\alpha$, define 
\[
(\alpha_1,\dots,\alpha_n)_g=C_{\alpha}\times2^{-\delta_{1g}\delta_{1n}}\times\prod_{i=1}^n\alpha_i!\times2^{|\alpha|},
\]  
then for $n>0$ and $\sum_i\alpha_i=3g-3+n$,
\[
\mbox{the dilaton equation}\qquad\quad (1,\alpha_1,\dots,\alpha_n)_g\,=\,(2g-2+n)(\alpha_1,\dots,\alpha_n)_g
\]
and for $n>0$ and $\sum_i\alpha_i=3g-2+n$,
\[
\mbox{the string equation}\qquad (0,\alpha_1,\dots,\alpha_n)_g\,=\,\sum_{\alpha_i\ne 0}(\alpha_1,\dots,\alpha_i-1,\dots,\alpha_n)_g.
\]
For the intersection number generating function 
\[
\mathbf{F}(\lambda,t_0,t_1,\dots)\,=\,\sum_{g=0}^{\infty}\lambda^{2g-2}\sum_{\{d_j\}}\,\langle\prod_{j=1}^{\infty}\tau_{d_j}\rangle_g\,\prod_{r\ge0}t_r^{n_r}/n_r!\,,
\]
with $n_r=\#\{j\mid d_j=r\}$, and
\[
\langle \tau_{d_1}\cdots\tau_{d_{n}}\rangle_g\,=\,\int_{\overline{\caM}_{g,n}}\prod_{j=1}^{n}\psi_j^{d_j},
\]
for $\sum_{j=1}^{n}d_j=3g-3+n$ and the product $\langle\tau_*\rangle$ otherwise zero, then the exponential $e^{\mathbf F}$ satisfies Virasoro algebra constraints, \cite[Theorem 6.1]{Mirwitt}.
\end{thrm}
\vspace{.1in}

\section{The organizational outline and reading guide.}  
The following outline combines \cite{Mirvol, Mirwitt} with the exposition of \cite[Chapter 9]{Wlcbms}.  The lectures are presented in the next section.  
\begin{itemize} 
  \item Teichm\"{u}ller spaces, moduli spaces, mapping class groups and the symplectic geometry.
  \begin{itemize}
	\item The Teichm\"{u}ller space $\caT_{g,n}$ and moduli space $\caM_{g,n}$; the Teichm\"{u}ller space $\caT_g\LL$ and moduli space $\caM_g\LL$ of prescribed length geodesic bordered hyperbolic surfaces; the augmented Teichm\"{u}ller space and Deligne-Mumford type compactifications; Dehn twists and the mapping class group $\mcg$ action.
	\item The WP symplectic geometry, \cite{Wlcbms}.
	\begin{itemize}
		\item The symplectic form $\omega=2\omega_{\tiny{\mbox{WP\ K\"ahler}}}$ and normalizations.
		\item Geodesic-length functions $\ell_{\alpha}$, Fenchel-Nielsen infinitesimal twist deformations $t_{\alpha}$ and the duality formula $\omega(\ ,t_{\alpha})=d\ell_{\alpha}$.
		\item Fenchel-Nielsen (FN) twist-length coordinates $(\tau_j,\ell_j)$ for Teichm\"{u}ller space and the formula $\omega=\sum_jd\ell_j\wedge d\tau_j$.  
	\end{itemize}
	\item The intermediate moduli space $\caM^{\gamma}_{g,n}$ of pairs $(R,\gamma)$ - a surface and a weighted multicurve $\gamma=\sum_jc_j\gamma_j$.
	\begin{itemize} 
		\item The covering tower
\[
\caT_{g,n}\longrightarrow \caM^{\gamma}_{g,n}\longrightarrow\caM_{g,n}.
\]
		\item The stabilizer subgroup $\operatorname{Stab}(\gamma)\subset\mcg$ for a weighted multicurve.  The $\mcg$ deck cosets for the covering tower.
		\item  Symplectic structures for $\caM^{\gamma}_{g,n}$ and $\caM_g(L)$.
	\end{itemize}
	\item The $\caT_{g,n}$ and $\caM^{\gamma}_{g,n}$ level sets of the total length $\ell=\sum_jc_j\ell_{\gamma_j}$.
\begin{lemma} \textup{\cite[Lemma 7.2]{Mirvol}.} \textup{Preparation for volume recursion and symplectic reduction.}  A finite cover of $\caM^{\gamma}_{g,n}$ is a fibered product of symplectic planes and lower dimensional moduli spaces.
\end{lemma}
  \end{itemize}
	\item The McShane-Mirzakhani length identity.
 		\begin{itemize} 
			\item The set $\mathcal B$ of homotopy classes rel boundary of simple arcs with endpoints on the boundary.  Classification of geodesic arcs normal to a boundary: simple geodesics normal to a boundary at each endpoint $\Longleftrightarrow$ disjoint pairs of boundary intervals $\Longleftrightarrow$ wire frames for pants.
			\item Birman-Series: simple geodesics have measure zero.	
			\item The rational exponential function $H$; the hyperbolic trigonometric functions $\caD$ and $\caR$; relations.
			\item The length identity.
\begin{thrm} \textup{\cite[Theorem 1.3 \& 4.2]{Mirvol} and \cite[Thrm. 1.8]{TanWZ}.} For a hyperbolic surface $R$ with boundaries $\beta_j$ with lengths $L_j$,
\[
L_1\,=\,\sum_{\alpha_1,\alpha_2}\,\caD(L_1,\ell_{\alpha_1}(R),\ell_{\alpha_2}(R))\,+\,
\sum_{j=2}^n\sum_{\alpha}\,\caR(L_1,L_j,\ell_{\alpha}(R)),
\]
where the first sum is over all unordered pairs of simple closed geodesics with $\beta_1,\alpha_1,\alpha_2$ bounding an embedded pair of pants, and the double sum is over simple closed geodesics with $\beta_1,\beta_j,\alpha$ bounding an embedded pair of pants.
\end{thrm}
			\item \cite[Section 8]{Mirvol} - Recognizing and understanding the identity as a smooth analog to a $\mcg$ fundamental domain.  Reducing to the action of smaller $\mcg$ groups.
		\end{itemize}
	\item The Mirzakhani volume recursion.
		\begin{itemize} 
			\item A covolume formula - writing a moduli integral as a length level set integral - the role of the intermediate moduli space $\caM_{g,n}^{\gamma}$ - \cite[Theorem 7.1]{Mirvol} and \cite[Theorem 9.5]{Wlcbms}.
			\item The application for the McShane-Mirzakhani identity.
				\begin{itemize}
					\item The integrals $\caA_{g,n}^{connected}$, $\caA_{g,n}^{disconnected}$ and $\caB_{g,n}$ of lower-dimensional moduli volumes.	 
					\item The corresponding connected and disconnected boundary pants configurations.

					\item Combining the length identity and covolume formula for the main result; see \cite[pg. 91 bottom, pg. 92]{Wlcbms}.
\begin{thrm} \textup{\cite[Section 5 and Theorem 8.1]{Mirvol}.}  For $(g,n)\ne (1,1), (0,3),$ the volume $V_{g,n}(L)$ satisfies
\[
\frac{\partial\ }{\partial L_1}L_1V_{g,n}(L)\,=\,\caA_{g,n}^{connected}(L)\,+\caA_{g,n}^{disconnected}(L)\,+\caB_{g,n}(L).
\]
\end{thrm}

					\item The integrals $\caA^*_*(L)$ and $\caB_*(L)$ are polynomials in boundary lengths with coefficients sums of special values of Riemann zeta; the coefficients are positive rational multiples of powers of $\pi$. 
   
				\end{itemize}
		\end{itemize}
	\item Symplectic reduction and the Duistermaat-Heckman theorem.
	\begin{itemize}
		 
		\item The WP kappa equation $\omega=2\pi^2\kappa_1$ on Deligne-Mumford \cite{Wlhyp}.
		\item The geometry and topology of the Teichm\"{u}ller space $\widehat{\caT}_g(L)=\widehat{\caT}_g(L_1,\dots,L_n)$ of hyperbolic surfaces with geodesic boundaries with points.

		\item A symplectic structure for $\widehat{\caT}_g(L)$ by summing on almost tight pants. The boundary length moment map $R\in \widehat{\caT}_g(L) \mapsto L^2/2\in \mathbb R_+^n$.  Twisting boundary points as Hamiltonian flows.  
		\item Symplectic reduction, \cite{Fsympred}.  $\caT_g(L)$ as the reduced space $\widehat{\caT}_g(L)/(S^1)^n$.  $\widehat{\caT}_g(L)$ as a principal $(S^1)^n$ bundle over $\caT_{g,n}(L)$ and the small $L$ equivalence $\caT_g(L)\approx\caT_{g,n}(0)$,  
\begin{equation*}
\xymatrix{ (S^1)^n\  \ar[r] & \widehat{\caT}_g(L) \ar[d] \\ & \caT_g(L)\approx \caT_{g,n}(0)\,.}
\end{equation*}
		\item Points on circles, cotangent $\mathbb C$-lines at punctures, circle bundles and homotopic structure groups.  The tautological cotangent line class $\psi$.
		\item $\mcg$ equivariant maps and quotients. Deligne-Mumford type compactifications and finite covers.  The elliptic stack $2$.
		\item The Duistermaat-Heckman normal form, \cite[Theorem\ 3.2]{Mirwitt}, \cite[Section 2.5]{SafMu} and  \cite[pg. 95]{Wlcbms},  
\[
2\omega_{\overline{\widehat{\caM}}_g(L)/(S^1)^n}\, \equiv\, 2\omega_{\overline{\caM}_{g,n}(0)}\ +\ \sum_j \frac{L_j^2}{2}\psi_j\ .
\]
		\item The consequence of symplectic reduction, \cite[Theorem 4.4]{Mirwitt}.
\begin{thrm} The volume polynomial $V_{g,n}(L)$ coefficients are $\overline{\caM}_{g,n}$ intersection numbers given as
\[
C_{\alpha}=\frac{2^{\delta_{1g}\delta_{1n}}}{2^{|\alpha|}\alpha!(3g-3+n-|\alpha|)!}
\,\int_{\overline{\caM}_{g,n}}\psi_1^{\alpha_1}\cdots\psi_n^{\alpha_n}\omega^{3g-3+n-|\alpha|},
\]
where $\psi_j$ is the Chern class for the cotangent line along the $j^{th}$ puncture, $\omega$ is the symplectic form and $\delta_{**}$ is the Kronecker indicator. 
\end{thrm}	
	\end{itemize}
	\item The pattern of intersection numbers.
 	\begin{itemize} 
		\item The general intersection number symbol 
\[	\langle \kappa_1^{d_0}\tau_{d_1}\cdots\tau_{d_n}\rangle_g\,=\,\int_{\overline{\caM}_{g,n}}\prod_{j=1}^n\psi_j^{d_j}\kappa_1^{d_0}
\]
and volume polynomial expansions. 
		\item Examples: $V_{1,1}(L)=\frac{\pi^2}{6}\,+\,\frac{L^2}{24}$, $V_{0,4}(L_1,L_2,L_3,L_4)=(4\pi^2+L_1^2+L_2^2+L_3^2+L_4^2)/2$ and $V_{2,1}(L)=\frac{1}{2211840}(L^2+4\pi^2)(L^2+12\pi^2)(5L^4+384\pi^2L^2+6960\pi^4)$.
		\item The partition function $\mathbf F=\sum_g\langle e^{\sum_jt_j\tau_j}\rangle_g$ and Virasoro constraint differential operators $\mathbf L_k$ \cite{Mirwitt}.  The partition function $\mathbf G=\sum_g\langle e^{s\kappa_1+\sum_jt_j\tau_j}\rangle_g$ and Virasoro constraint differential operators $\mathbf V_k$ \cite[Theorem 1.1]{SafMu}. 
		
		\item The Do {\em remove a boundary} relation
\[
\frac{\partial V_{g,n+1}}{\partial L_{n+1}}(L,2\pi i)\,=\,2\pi i(2g-2+n)V_{g,n}(L), \qquad\cite{Dothe,DoNo}.
\]
		\item The Manin-Zograf volumes generating function \cite{MZ}.  The punctures asymptotic - for positive constants $c,\, a_g$,  genus fixed and large $n$, then
\[
V_{g,n}\,=\,c^nn!\,n^{(5g-7)/2}(a_g\,+\,O(1/n)), \qquad\cite{MZ}.
\]
		\item The Schumacher-Trapani genera asymptotic - for $n$ fixed, there are positive constants, then
\[
c_1^{\,g}(2g)!\,<\,V_{g,n}\,<\,c_2^{\,g}(2g)!, \qquad\cite{Gru, SchTr}.  
\]
		\item The Zograf conjecture 
\[
V_{g,n}\,=\,(4\pi^2)^{2g+n-3}(2g+n-3)!\frac{1}{\sqrt{g\pi}}\big(1\,+\,\frac{c_n}{g}\,+\,O(1/g^2)\big),
\]
for fixed $n$ and $g$ tending to infinity, \cite{Zogasym}.  Expected values of geometric invariants \cite{Mirasymp}. 
 	\end{itemize}
\end{itemize}

\section{The lectures.}

The following is an exposition of Mirzakhani's proof of the Witten-Kontsevich theorem, including the immediate background material on Teichm\"{u}ller theory, moduli space theory and on symplectic reduction.  The lectures are presented as from a graduate text - the exposition follows the development of concepts, and does not consider the historical development of the material. The goals are general treatment of the material and overall understanding for the reader.  In places, the approaches of several authors are combined for a simpler treatment of the material.  Only immediate references to the literature are included. The reader should consult the literature for the historical development, for complete references and for consequences of the material.

\begin{enumerate}
  \item \hyperlink{lec1}{\bfseries The background and overview. }
  \item \hyperlink{lec2}{\bfseries The McShane-Mirzakhani identity.}
  \item \hyperlink{lec3}{\bfseries The covolume formula and recursion.}
  \item \hyperlink{lec4}{\bfseries Symplectic reduction, principal $S^1$ bundles and the normal form.}
  \item \hyperlink{lec5}{\bfseries The pattern of intersection numbers and Witten-Kontsevich.}
\end{enumerate}

\noindent {\hypertarget{lec1}{{\bfseries Lecture 1: The background and overview.}}}

\noindent {\bfseries General background.}

By Uniformization, for a surface of negative Euler characteristic, a conformal structure is equivalent to a complete hyperbolic structure.  We consider Riemann surfaces $R$ of finite topological type with hyperbolic metrics, possibly with punctures and geodesic boundaries, if boundaries are non empty.  Fix a topological reference surface $F$, and consider a marking, an orientation preserving homeomorphism $f:F\rightarrow R$ up to homotopy.  If boundaries are non empty, homotopy is rel boundary setwise.  Write $\caT$ for the Teichm\"{u}ller space of $R$ - the space of equivalence classes of pairs $\{(f,R)\}$, where pairs are equivalent if there is a homotopy mapping triangle with a conformal map (a hyperbolic isometry) between Riemann surfaces.

\begin{figure}[htbp] 
  \centering
  \includegraphics[bb=0 0 506 124,width=3.75in,height=0.858in,keepaspectratio]{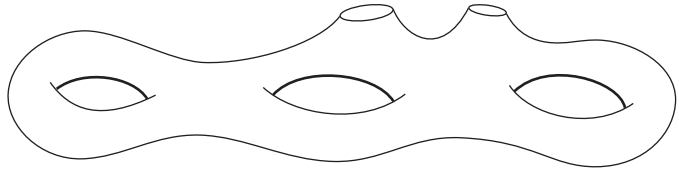}
  \caption{A genus $3$ surface with $2$ geodesic boundaries.}
  \label{fig:genus3bdr}
\end{figure}

We consider the following Teichm\"{u}ller spaces $\caT$: $\Tg$ - for compact genus $g$ surfaces; $\TgL$ - for genus $g$ surfaces with labeled geodesic boundaries of prescribed lengths $L=(L_1,\dots,L_n)$; $\Tgn$ - for genus $g$ surfaces with $n$ labeled punctures.  In the case of $\TgL$, homotopies of surfaces are rel boundaries setwise.  $\Tg$ and $\Tgn$ are complex manifolds, while $\TgL$ is a real analytic manifold.

A non trivial, non puncture peripheral, free homotopy class $\alpha$ on $F$ has a unique geodesic representative for $f(\alpha)$ on the surface $R$ - the geodesic length $\lla(R)$ provides a natural function on Teichm\"{u}ller space.  Collections of geodesic-length functions provide local coordinates and global immersions to Euclidean space for $\caT$.   The differential of geodesic-length for a simple curve is nowhere vanishing.  At each point of $\caT$, the differentials of geodesic-lengths of simple curves are dense in the cotangent bundle.

A surface  can be cut open on a simple closed geodesic - the boundaries are isometric circles.  Since a neighborhood of a simple geodesic has an $S^1$ symmetry, the boundaries can be reassembled with a relative rotation to form a new hyperbolic structure. The deformation is the Fenchel-Nielsen (FN) twist. 
\begin{figure}[htbp] 
  \centering
  \includegraphics[bb=0 0 529 251,width=3.3in,height=1.56in,keepaspectratio]{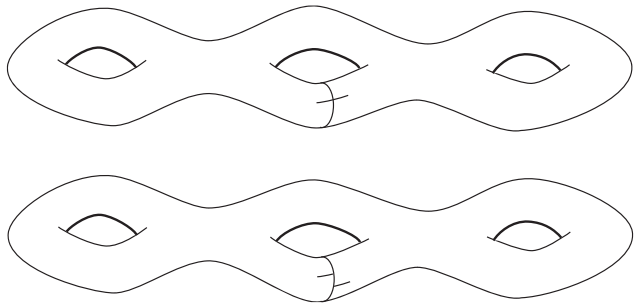}
  \caption{A positive Fenchel-Nielsen twist deformation.}
  \label{fig:genus3t}
\end{figure}
The infinitesimal deformation for unit speed hyperbolic displacement of initial adjacent points, is the Fenchel-Nielsen infinitesimal twist vector field 
$t_{\alpha}$ on $\caT$.  (A positive twist corresponds to displacing to the right when crossing the geodesic.)

Geodesic boundaries of hyperbolic surfaces of common length can be assembled to form new surfaces.   Given boundary reference points, the relative rotation is measured in terms of arc length.  A hyperbolic pair of pants is a genus zero surface with three geodesic boundaries.  For pants, boundary reference points are provided by considering the unique orthogonal connecting geodesics between boundaries.  At a gluing, the common boundary length $\ell$ and relative rotation, the FN twist parameter $\tau$, are unrestricted parameters ($\tau$ is defined by continuation from an initial configuration).  The length $\ell$ varies in $\mathbb R_{>0}$ and twist $\tau$ varies in $\mathbb R$.  Each finite topological type hyperbolic surface can be assembled from pairs of pants.  

\begin{thrm*}\textup{Fenchel-Nielsen coordinates.}  Fixing the topological type of a pants decomposition and an initial configuration, the FN parameters $\prod_{j=1}^{3g-3+n}(\ell_j,\tau_j)$ define a real analytic equivalence of $\caT$ to $\prod_{j=1}^{3g-3+n}\mathbb R_{>0}\times\mathbb R$.
\end{thrm*}

The Weil-Petersson (WP) metric is K\"{a}hler.  The symplectic geometry begins with the symplectic form $\omega=2\omega_{\tiny{\mbox{WP\ K\"ahler}}}$ and the basic twist-length duality 
\[
\omega(\ ,t_{\alpha})=d\ell_{\alpha}.
\]  
It follows from the Lie derivative equation $L_X\omega(\ ,\ )=d\omega(X,\ ,\ )+d(\omega(X,\ ))$ that the form $\omega$ is invariant under all twist flows.  It follows that geodesic-length functions are Hamiltonian potentials for $FN$ infinitesimal twists. Symmetry reasoning shows that $\ell$ and $\tau$ provide action-angle coordinates for $\omega$.  

\begin{thrm*}\textup{(W), \cite{Wlcbms}.\ The $d\ell\wedge d\tau$ formula.} The WP symplectic form is
\[
\omega\,=\,\sum_{j=1}^{3g-3+n}\,d\ell_j\wedge d\tau_j.
\]
\end{thrm*}

Frontier spaces are adjoined to $\caT$ corresponding to allowing $\ell_j=0$ with the FN angle $\theta_j=2\pi\tau_j/\ell_j$ then undefined (in polar coordinates, angle is undefined at the origin).  The vanishing length describes pairs of pants with corresponding boundaries represented by punctures - the equation $\lla=0$ describes hyperbolic structures with $\alpha$ represented by pairs of punctures.  For a subset of indices $J\subset\{1,\dots,3g-3+n\}$, the $J$-null stratum is $\caS(J)\,=\,\{R\mbox{ degenerate}\mid \ell_j(R)=0\mbox{ iff } j\in J\}$.  Each null stratum is a product of lower dimensional Teichm\"{u}ller spaces.  A basis of neighborhoods in $\caT\cup\caS(J)$ is defined in terms of the parameters $(\ell_k,\theta_k,\ell_j)_{k\notin J,\, j\in J}$. 

The augmented Teichm\"{u}ller space is the stratified space 
\[
\Tbar=\caT\cup_{\tiny{pants\ decompositions\ }\caP}\cup_{J\subset\caP}\,\caS(J).
\]  
The space $\Tbar$ is also described as the Chabauty topology closure of the discrete faithful type-preserving representations of $\pi_1(F)$ into $PSL(2;\mathbb R)$, modulo $PSL(2;\mathbb R)$ conjugation.  The augmentation construction is valid for $\Tg,\TgL$ and $\Tgn$.  $\Tbar$ is a Baily-Borel type partial compactification.  $\Tbar$ is never locally compact.  The $d\ell\wedge d\tau$ formula provides for the extension of the symplectic structure to the augmented Teichm\"{u}ller space $\Tbar$.  Each strata is symplectic.   

The mapping class group ($\mcg$) $Homeo^+(F)/Homeo_0(F)$ acts on markings by precomposition and thus acts on $\caT$.  For $Homeo^+(F)$ we consider type-preserving (boundary point, boundary curve), boundary label preserving, orientation preserving homeomorphisms.  $Homeo_0(F)$ is the normal subgroup of elements homotopic to the identity rel boundary setwise.  A Dehn twist is a homeomorphism that is the identity on the complement of a tubular neighborhood of a simple closed curve, non trivial in homotopy, and rotates one boundary of the tubular neighborhood relative to the other.  Dehn twist classes generate $\mcg$. $\mcg$ acts properly discontinuously on $\caT$ and by biholomorphisms for $\Tg$ and $\Tgn$.  Except for a finite number of topological types the action is effective.  Finite $\mcg$ subgroups act with fixed points.  $\mcg$ acts on the stratified space $\Tbar$.  Bers observed that there are constants $b_{g,n}$, depending on topological type, such that a genus $g$, $n$ punctured hyperbolic surface has a pants decomposition with seam lengths at most $b_{g,n}$.   It follows that the domain $\{\ell_j\le b_{g,n},\, 0<\tau_j\le \ell_j\}$ in FN coordinates is a rough fundamental set - each $\mcg$ orbit intersects the domain a bounded positive number of times.  $\Tbar/\mcg$ is a compact real analytic orbifold; $\Tbarg/\mcg$ and $\Tbargn/\mcg$ are topologically the Deligne-Mumford stable curve compactifications of $\Mg$ and $\Mgn$.  The Bers fundamental set observation combines with the $d\ell\wedge d\tau$ formula to provide that the WP volume of $\caM$ is finite.

The Bers fiber space $\caC$ (specifically $\Cg$ and $\Cgn$) is the complex disc holomorphic bundle over $\caT$ with fiber over $\{(f,R)\}$ the universal cover $\widetilde R$. A point on a fiber can be considered as a puncture and determines a curve from basepoint for the  fundamental group - so $\Cg\approx\caT_{g,1}$ and $\Cgn\approx\caT_{g,n+1}$. An extension $\mcg_{\mathcal C}$ of $\mcg(F)$ by the fundamental group $\pi_1(F)$ acts properly discontinuously and holomorphically on $\mathcal C$.  The group  $\mcg_{\mathcal C}$ is isomorphic to $\mcg_{g,n+1}$.  For the epimorphism from $\mcg_{\mathcal C}$ to $\mcg(F)$, the first group acts equivariantly on the fibration of $\caC$ over $\caT$.   The resulting map $\pi:\caC/\mcg\rightarrow\caT/\mcg$ describes an orbifold bundle, the universal curve, with orbifold fibers - the fibers are Riemann surfaces modulo their full automorphism group.  Manifold finite local covers and the quotient can be described by starting with a surface with locally maximal symmetries and introducing a local trivialization of the bundle by canonical (extremal, harmonic) maps of surfaces.

The augmentation construction applies to the Bers fiber space to give $\Cbar$.  $\mcg_{\mathcal C}$ acts on the stratified space $\Cbar$.        
\begin{figure}[htbp] 
  \centering
  \includegraphics[bb=0 0 228 196,width=2.5in,height=2.15in,keepaspectratio]{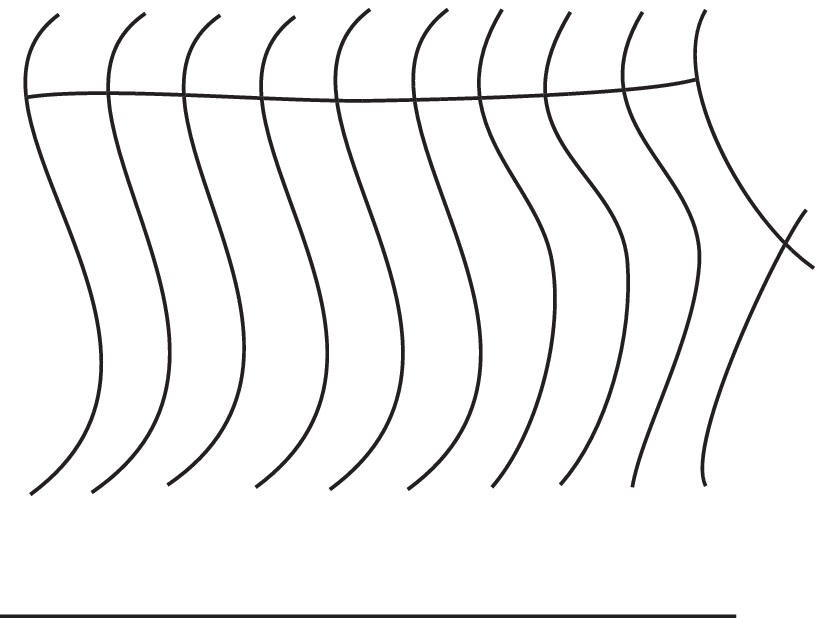}
  \caption{A puncture section of the universal curve $\Cbar/\mcg_{\mathcal C}$ over  $\Tbar/\mcg(F)$.}
  \label{fig:univcurvs}
\end{figure}
\noindent The augmentation quotient $\Cbar/\mcg$ is an orbifold and almost an orbifold bundle over $\Mbar$ - at a node (a pair of punctures) of a Riemann surface, the fiber becomes vertical - the local model of the fibration is the germ at the origin of the projection $\{(z,w)\}\rightarrow \{t=zw\}$, the family of complex hyperbolas. The turning of the fibers of the almost orbifold bundle $\Cbar/\mcg\rightarrow\Mbar$ is measured by line bundles on $\Mbar$.  The family of tangent $\mathbb C$-lines to the fibers $(Ker\,d\pi)$ is the tangent bundle along a smooth Riemann surface and the relative dualizing sheaf along a noded Riemann surface.  The hyperbolic metrics of the individual fibers provide a line bundle metric for $(Ker\,d\pi)$ on $\Cbar$, that although not smooth is sufficiently regular for calculation of the Chern form $\mathbf c_1$.  The kappa forms/cohomology classes $\kappa_k=\int_{\pi^{-1}(\{R\})}\mathbf c_1^{k+1}$ given by integration over fibers are basic to moduli geometry.  The geometry and algebra of the kappa classes is studied in the Carel Faber lectures.  Explicit calculation of the Chern form and integration provides the following.
\begin{figure}[htbp] 
  \centering
  \includegraphics[bb=0 0 235 176,width=2.5in,height=1.87in,keepaspectratio]{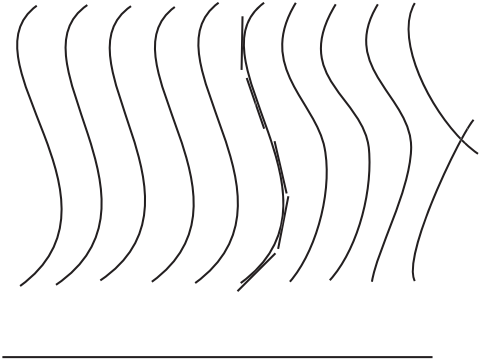}
  \caption{Tangents along a fiber of the universal curve $\Cbar/\mcg_{\mathcal C}$.}
  \label{fig:univcurvt}
\end{figure}
\begin{thrm*}\textup{(W), \cite{Wlhyp}.} \label{wpkap} For the hyperbolic metrics on fibers, $2\pi^2\kappa_1=\omega$ pointwise on $\caM$ and in cohomology on $\Mbar$. 
\end{thrm*}

A conformal structure has a unique extension to fill in a puncture.  A labeled puncture defines a section $s$ of $\Cbar_{g,n}/\mcg\rightarrow\Mbargn$.  A section satisfies $\pi\circ s=id$, differentiating gives $d\pi\circ ds=d\,id$.  At a node, $d\pi$ vanishes in the node opening direction - for $\pi(z,w)=t$ then $d\pi=wdz\,+\,zdw$ vanishes at the origin.  Sections of $\Cbar/\mcg$ over $\Mbar$ are consequently disjoint from nodes.  Along a puncture section $s:\Mbar\rightarrow\Cbar/\mcg$, we consider the family of tangent lines $(Ker\, d\pi)|_s$ or the dual family  $(Ker\, d\pi)^*|_s$.  In the Carel Faber lectures, the Chern class is denoted as $K$.
 -The pullback to $\Mbar$ by a puncture section $s$ of the dual family $(Ker\, d\pi)^*$ is the moduli geometry canonical psi class $\psi$ .-  

In these lectures, to emphasize concepts and the underlying geometry, we will at times informally interchange a line bundle and its Chern class, informally refer to the moduli space as a manifold, the universal curve as a fiber bundle, and at times refer to the open moduli space when actually the augmentation quotient is required.  Our goal is to discuss the central matters.  In spite of the informal approach, an experienced reader will find that the treatment is complete.  

Basic references for the above material are \cite{Busbook} and \cite{Wlcbms}.

\noindent {\bfseries Volume results overview.}    

Mirzakhani shows that the WP volume $V_{g,n}(L)\,=\,V(\MbargL)$ is a polynomial in $L$, with coefficients given by the intersection numbers of powers of $\kappa_1$ and powers of $\psi$.  She further shows that her recursion for determining the volume polynomials satisfies the defining relations for the Witten-Kontsevich conjecture.  The following theorems are the immediate results of the two papers. 

\begin{thrm*} \textup{The WP volume polynomials.}\label{first}  The volume polynomials are determined recursively from the volume polynomials of smaller total degree, \cite[Formula (5.1) \& Theorem 8.1]{Mirvol}.   The volume $V_{g,n}(L_1,\dots,L_n)$ of the moduli space of genus $g$, $n$ boundaries, hyperbolic surfaces with boundary lengths $L=(L_1,\dots,L_n)$ is a polynomial
\[
V_{g,n}(L)\,=\,\sum_{\stackrel{\alpha}{|\alpha|\le 3g-3+n}} C_{\alpha}\,L^{2\alpha},
\]
for multi index $\alpha=(\alpha_1,\dots,\alpha_n)$ and where $C_{\alpha}>0$ lies in $\pi^{6g-6+2n-2|\alpha|}\mathbb Q$, \cite[Theorems 1.1 \& 6.1]{Mirvol}.  The coefficients are intersection numbers given as 
\[
C_{\alpha}=\frac{2^{\delta_{1g}\delta_{1n}}}{2^{|\alpha|}\alpha!(3g-3+n-|\alpha|)!}
\,\int_{\overline{\caM}_{g,n}}\psi_1^{\alpha_1}\cdots\psi_n^{\alpha_n}\omega^{3g-3+n-|\alpha|},
\]
where $\psi_j$ is the Chern class for the cotangent line along the $j^{th}$ puncture, $\omega$ is the symplectic form, $\alpha!=\prod_{j=1}^n\alpha_j!$, and $\delta_{**}$ is the Kronecker indicator delta, \cite[Theorem 4.4]{Mirwitt}. 
\end{thrm*}
\begin{thrm*} \cite[Theorems 6.3 \& 6.4]{Mirvol}.  \textup{Recursive relations for the volume polynomial leading coefficients.} For a multi index $\alpha$, define 
\[
(\alpha_1,\dots,\alpha_n)_g=C_{\alpha}\times2^{-\delta_{1g}\delta_{1n}}\times\prod_{i=1}^n\alpha_i!\times2^{|\alpha|},
\]
then for $n>0$ and $\sum_i\alpha_i=3g-3+n$,
\[
\mbox{the dilaton equation}\qquad\quad (1,\alpha_1,\dots,\alpha_n)_g\,=\,(2g-2+n)(\alpha_1,\dots,\alpha_n)_g
\]
and for $n>0$ and $\sum_i\alpha_i=3g-2+n$,
\[
\mbox{the string equation}\qquad (0,\alpha_1,\dots,\alpha_n)_g\,=\,\sum_{\alpha_i\ne 0}(\alpha_1,\dots,\alpha_i-1,\dots,\alpha_n)_g.
\]
For the intersection number generating function
\[
\mathbf{F}(\lambda,t_0,t_1,\dots)\,=\,\sum_{g=0}^{\infty}\lambda^{2g-2}\sum_{\{d_j\}}\,\langle\prod_{j=1}^{\infty}\tau_{d_j}\rangle_g\,\prod_{r\ge0}t_r^{n_r}/n_r!\,,
\]
with $n_r=\#\{j\mid d_j=r\}$, and
\[
\langle \tau_{d_1}\cdots\tau_{d_{n}}\rangle_g\,=\,\int_{\overline{\caM}_{g,n}}\prod_{j=1}^{n}\psi_j^{d_j},
\]
for $\sum_{j=1}^{n}d_j=3g-3+n$ and the product $\langle\tau_*\rangle$ otherwise zero, then the exponential $e^{\mathbf F}$ satisfies Virasoro algebra constraints, \cite[Theorem 6.1]{Mirwitt}. 
\end{thrm*}

A fine structure for volumes is suggested by the Zograf conjecture
\[
V_{g,n}\,=\,(4\pi^2)^{2g+n-3}(2g+n-3)!\frac{1}{\sqrt{g\pi}}\big(1\,+\,\frac{c_n}{g}\,+\,O(1/g^2)\big),
\]
for fixed $n$ and $g$ tending to infinity.

As an application of the method for recursion of volumes and intersection numbers, Do derives a {\em remove a boundary} relation for the volume polynomials
\[
\frac{\partial V_{g,n+1}}{\partial L_{n+1}}(L,2\pi i)\,=\,2\pi i(2g-2+n)V_{g,n}(L), \qquad\cite{Dothe}.
\]
The relation gives the compact case volume $V_g$.

\noindent {\bf Statement of the volume recursion} \cite[Sec. 5]{Mirvol}{\bf. }\hypertarget{volrecur}   The WP volume $V_g(L_1,\dots,L_n)$ of the moduli space $\caT_g(L_1,\dots,L_n)/\mcg$ is a symmetric function of boundary lengths as follows.
\begin{itemize}
  \item For $L_1,L_2,L_3 \ge 0$, formally set\\
\[
 V_{0,3}(L_1,L_2,L_3)=1
\]
and\\ 
\[
V_{1,1}(L_1)=\frac{\pi^2}{12}+\frac{L_1^2}{48}.
\]
  \item For $L=(L_1,\dots,L_n)$, let $\widehat L=(L_2,\dots,L_n)$ and for $(g,n)\ne (1,1)$ or $(0,3)$, the volume satisfies\\
\[ 
\frac{\partial\ }{\partial L_1}L_1V_g(L)=\caA_g^{con}(L)+\caA_g^{dcon}(L)+\caB_g(L)
\]
where
\[
\caA_g^*(L)=\frac12\int_0^{\infty}\int_0^{\infty}\widehat\caA_g^*(x,y,L)\,xy\,dxdy
\]
and
\[
\caB_g(L)=\int_0^{\infty}\widehat\caB_g(x,L)\,x\,dx.
\]
The quantities $\widehat\caA_g^{con},\,\widehat\caA_g^{dcon}$ are defined in terms of the function 
\[
H(x,y)=\frac{1}{1+e^{\frac{x+y}{2}}}+\frac{1}{1+e^{\frac{x-y}{2}}}
\]
and moduli volumes for subsurfaces
\[
\widehat\caA_g^{con}(x,y,L)=H(x+y,L_1)V_{g-1}(x,y,\widehat L)
\]
and surface decomposition sum
\[
\quad\quad\quad\widehat\caA_g^{dcon}(x,y,L)=\sum\limits_{\stackrel{g_1+g_2=g}{I_1\amalg I_2=\{2,\dots,n\}}} H(x+y,L_1)
V_{g_1}(x,L_{I_1})V_{g_2}(y,L_{I_2}),
\]
where in the second sum only decompositions for pairs of hyperbolic structures are considered and the unordered sets $I_1,I_2$ provide a partition.  The third quantity $\widehat\caB_g$ is defined by the sum
\[
\quad\quad\quad\quad \frac12 \sum_{j=1}^n\big(H(x,L_1+L_j)+H(x,L_1-L_j)\big)\,V_g(x,L_2,\dots,\widehat{L_j},\dots,L_n),
\]
where $L_j$ is omitted from the argument list of $V_g$.
\end{itemize}
{\em The basic point:} the volume $V_g(L_1,\dots,L_n)$ is an appropriate integral of volumes for surfaces formed with {\em one} fewer pairs of pants.

\noindent\hypertarget{lec2}{{\bfseries Lecture 2: The McShane-Mirzakhani identity.}}

In 1991, Greg McShane discovered a universal identity for a sum of lengths of simple geodesics for a once punctured torus, \cite{McSh}.  A generalization of the identity serves as the analog of a partition of unity for the action of the mapping class group.  The identity enables reduction of the action to the actions of smaller mapping class groups.  Consideration of the identity begins with a surface with geodesic boundaries and a study of arcs from the boundary to itself.  

Introduce $\caB$, the set of non trivial free homotopy classes of simple curves from the boundary to the boundary, homotopy rel the boundary.  We illustrate the approach by considering simple curves with endpoints on a common boundary $\beta$; the analysis is similar for simple curves connecting distinct boundaries.  Each homotopy class contains a unique shortest geodesic, orthogonal to $\beta$ at end points - refer to these geodesics as {\em ortho boundary geodesics}. If the surface is doubled across its boundary, then the ortho boundary geodesics double to simple closed geodesics.  

The set $\caB$ is in bijection to the set of topological pants embedded in the surface with $\beta$ as one boundary - refer to these pants as $\beta$-cuff pants.  First note that the endpoints of an ortho boundary geodesic $\gamma,[\gamma]\in\caB$ are distinct.  A small neighborhood/thickening of $\gamma\cup\beta$ is the corresponding topological pair of pants.  Geometrically, the curve $\gamma$ separates $\beta$ into proper sub arcs; the union of each sub arc with $\gamma$ is a simple curve, that defines a free homotopy class containing a unique geodesic. The corresponding geometric pair of pants $\caP$ has boundaries $\beta$ and the two determined geodesics.  We will see below that a geometric pair of pants contains a unique ortho boundary geodesic. The unions of ortho boundary geodesics and $\beta$ are the spines, the wire frames, for the embedded geometric $\beta$-cuff pants.

We now describe how the behavior of geodesics emanating orthogonally from $\beta$ defines a Cantor subset of $\beta$.  The Cantor set will have measure zero and the length identity is  simply the sum of lengths of the complementary intervals.   The following description follows the analysis by Tan-Wong-Zhang, \cite{TanWZ}.

Consider the maximal continuations of geodesics emanating orthogonally from $\beta$ - refer to these geodesics as {\em ortho emanating geodesics}. In addition to the ortho boundary geodesics, there are three types of ortho emanating geodesics: non simple, simple infinite length and simple crossing the boundary obliquely at a second endpoint.  We will see that the types are detected by considering initial segments in a pair of pants.

Consider the geometric pants $\caP$, obtained from an ortho boundary geodesic $\gamma$ (see Figure \ref{fig:maingap}).  The boundaries are $\beta$ and the two defined geodesics $\alpha$ and $\lambda$ (in the special case $(g,n)=(1,1)$ then $\alpha=\lambda$).  A {\em spiral} is an infinite simple geodesic ray that accumulates to a simple closed geodesic.  Two ortho emanating geodesics are spirals with accumulation set $\alpha$ and two are spirals with accumulation set $\lambda$.  The two spirals accumulating to a boundary wind in opposite directions around the boundary.  The $\caP$ {\em main gaps} are the two disjoint subarcs of $\beta$ that each contain in their interior an endpoint of $\gamma$ and have spiral initial points as endpoints. The main gaps will be the components of the Cantor set complement corresponding to the pants $\caP$.  From the geometry of pants, geodesics ortho emanating from the main gaps are either the spiral endpoints, $\gamma$, non simple with self intersection in $\caP$ or simple crossing $\beta$ obliquely at a second endpoint.  The complement in $\beta$ of the main gaps are a pair of open intervals.  From the geometry of pants, for a given open interval all ortho emanating geodesics exit the pants by crossing one of the boundaries $\alpha$ or $\lambda$.  For a given open interval, the initial segments in $\caP$ are simple and these geodesics are classified by their subsequent behavior elsewhere on the surface, by their behavior on some other pair of pants. The ortho boundary geodesic connecting $\beta$ to $\alpha$ is contained in one of the open intervals, and the ortho boundary geodesic from $\beta$ to $\lambda$ is contained in the other.  A pair of pants has an equatorial reflection, stabilizing each boundary.  The equatorial reflection acts naturally on the decomposition of $\beta$, interchanging or stabilizing elements.

\begin{figure}[htbp] 
  \centering
  \includegraphics[bb=0 0 507 397,width=3in,height=2.35in,keepaspectratio]{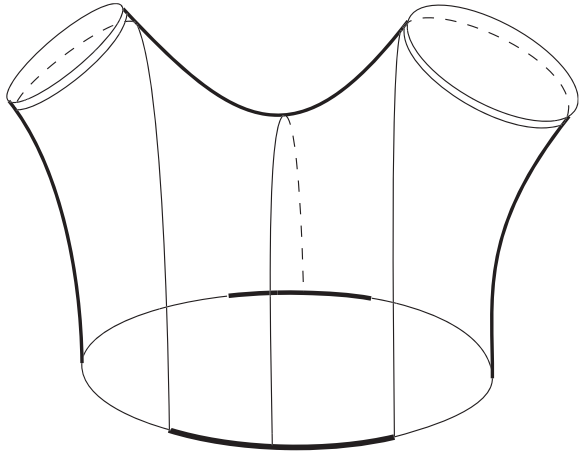}
  \caption{A pair of pants with equators, main gaps, two spirals and an orthoboundary geodesic $\gamma$.}
  \label{fig:maingap}
\end{figure}

In the above, associated to a main gap are the ortho emanating geodesics that self intersect in the pants, and the simple geodesics that obliquely cross $\beta$ a second time - refer to the second type geodesics as {\em boundary oblique}.   The next observation is that the associations can be reversed, the associations define bijections between main gaps and  geodesics with particular behaviors on the surface.  

A boundary oblique geodesic and $\beta$ form a crooked wire frame that determines a pair of pants, similar to an ortho boundary geodesic determining a pair of pants.  Boundary oblique geodesics come in continuous families with each family limiting to an ortho boundary geodesic.  A family and its limit determine the same pair of pants.  The initial points (the $\beta$ orthogonal points) of family elements lie in a common main gap interval - this observation reverses the association of segments of main gaps to boundary oblique geodesics.  

Next we describe reversing the association of segments of main gaps to non simple ortho emanating geodesics.  The first self intersection of such a geodesic is contained in a unique embedded pair of pants.  To see this, consider the lasso subarc beginning at $\beta$ and ending where the geodesic passes through its first self intersection point a second time.  The boundary of a small neighborhood/thickening of the lasso is the union of a simple closed curve and an element of $\caB$.  A geometric argument shows that the self intersection point is contained in the pants determined by the element of $\caB$ and the lasso initial point lies in the main gap for the pants.  This observation reverses the association of segments of main gaps to non simple ortho emanating geodesics.

\begin{figure}[htbp] 
  \centering
  \includegraphics[bb=0 0 507 401,width=2in,height=1.58in,keepaspectratio]{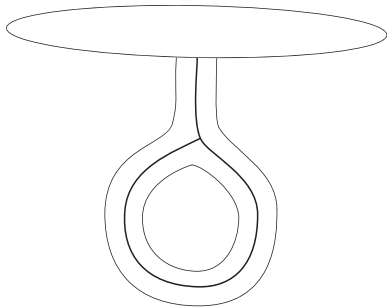}
  \caption{A small neighborhood of a lasso.}
  \label{fig:lasso}
\end{figure}

We recall that non simple with interior intersection is an open condition on the space of geodesics and an open condition on the space of ortho emanating geodesics.  By considering the double of the surface, simple with all boundary intersections orthogonal is a closed condition on the space of geodesics.  The set {\em simple with orthogonal single boundary intersection} is a Cantor set.  The classification of ortho emanating geodesics is complete.

\begin{thrm*}\textup{(Tan-Wong-Zhang \cite{TanWZ}, and Mirzakhani \cite{Mirvol}, all following McShane \cite{McSh}.)} There is a Cantor set partition of boundary points by the behavior of ortho emanating geodesics:
\begin{multline*}
\beta\,=\,\{\mbox{simple with orthogonal single boundary intersection}\}\\ \,
\cup\,\{\mbox{ortho boundary geodesics}\}\,\cup\,\{\mbox{simple boundary oblique geodesics}\}\,\cup\,\{\mbox{non simple}\}.
\end{multline*}
\end{thrm*}
     
An $\epsilon$-neighborhood of the simple complete geodesics, orthogonal to the boundary at intersections, is a countable union of thin corridors.  In the universal cover the corridors are described by reduced bi infinite words in the fundamental group.  By analyzing the number and width of corridors, Birman-Series show that the set is very thin.

\begin{thrm*}\textup{(Birman-Series, \cite{BS}.)\ Simple geodesics have measure zero.} The set $\caS$ of simple complete geodesics, orthogonal to the boundary at intersections, has Hausdorff dimension $1$.  The intersection of $\caS$ and the boundary has Hausdorff dimension and measure $0$.
\end{thrm*}

The basic summand for the length identity is a rational exponential function. Define the function $H$ on $\mathbb R^2$ by
\begin{equation}\label{hdef}
H(x,y)=\frac{1}{1+e^{\frac{x+y}{2}}}+\frac{1}{1+e^{\frac{x-y}{2}}}
\end{equation}    
and the corresponding functions $\caD,\caR$ on $\mathbb R^3$ by
\begin{equation}
\begin{split}
\caD(x,y,z)\,&=\,2\log\Bigg(\frac{e^{\frac{x}{2}}\,+\,e^{\frac{y+z}{2}}}{e^{\frac{-x}{2}}\,+\,e^{\frac{y+z}{2}}}\Bigg)  \quad\mbox{and}\\
\caR(x,y,z)\,&=\,x\,-\,\log\Bigg(\frac{\cosh\frac{y}{2}\,+\cosh\frac{x+z}{2}}{\cosh\frac{y}{2}\,+\cosh\frac{x-z}{2}} \Bigg).
\end{split}
\end{equation}
The functions $\caD$ and $\caR$ are related to $H$ as follows,
\begin{equation}\label{drdef}
\begin{split}
\frac{\partial\ }{\partial x}\caD(x,y,z)&=H(y+z,x),\quad \caD(0,0,0)=0\quad \mbox{and}\\
 2\,\frac{\partial\ }{\partial x}\caR(x,y,z)&=H(z,x+y)+H(z,x-y),\quad \caR(0,0,0)=0.
\end{split}
\end{equation}

\begin{thrm*} \textup{\cite[Theorem 1.3 \& 4.2]{Mirvol} and \cite[Thrm. 1.8]{TanWZ}.\label{lenid} The Mirzakhani-McShane identity.} For a hyperbolic surface $R$ with boundaries $\beta_j$ with lengths $L_j$,
\[
L_1\,=\,\sum_{\alpha_1,\alpha_2}\,\caD(L_1,\ell_{\alpha_1}(R),\ell_{\alpha_2}(R))\,+\,
\sum_{j=2}^n\sum_{\alpha}\,\caR(L_1,L_j,\ell_{\alpha}(R)),
\]
where the first sum is over all unordered pairs of simple closed geodesics with $\beta_1,\alpha_1,\alpha_2$ bounding an embedded pair of pants, and the double sum is over simple closed geodesics with $\beta_1,\beta_j,\alpha$ bounding an embedded pair of pants.
\end{thrm*}
\begin{proof}
By the above Theorems, $\ell_{\beta}$ equals the sum over embedded $\beta$-cuff pants of main gap lengths and the counterpart lengths for double boundary cuff pants.  To find the main gap lengths, begin with a formula for the lengths of the complementary intervals.   In a pair of pants, the ortho boundary geodesic $\delta_{\alpha}$ from $\beta$ to $\alpha$ bisects a complementary interval (see Figure \ref{fig:maingap}).  Let $\beta_{\alpha}$ be one of the resulting half intervals.  The segment $\beta_{\alpha}$ has the geodesic $\delta_{\beta}$ emanating at one end and a spiral $\sigma$ to $\alpha$ emanating at the other end.  In the universal cover, consider contiguous lifts $\widetilde{\delta_{\alpha}}$, $\widetilde{\beta_{\alpha}}$ and $\widetilde{\sigma}$.  The three lifts and a half infinite ray lift $\widetilde{\alpha}$ of $\alpha$, combine to form a quadrilateral 
$\widetilde{\alpha},\widetilde{\delta_{\alpha}},\widetilde{\beta_{\alpha}},\widetilde{\sigma}$ 
with angles $\pi/2,\pi/2,\pi/2$ and $0$ between $\widetilde{\sigma}$ and $\widetilde{\alpha}$.  By hyperbolic trigonometry of quadrilaterals \cite{Busbook}, it follows that
\[
\tanh \ell_{\beta_{\alpha}}\,=\,\operatorname{sech} \delta_{\alpha}\,=\,\frac{\sinh (\ell_{\beta}/2)\,\sinh (\ell_{\alpha}/2)}{\cosh (\ell_{\lambda}/2)\,+\,\cosh(\ell_{\beta}/2)\,
\cosh(\ell_{\alpha}/2)}\,.
\]
The complementary interval length is $2\ell_{\beta_{\alpha}}=\ell_{\beta}-\caR(\ell_{\beta},\ell_{\lambda},\ell_{\alpha})$.  The formula for main gap lengths now follows from the general relation $\caR(x,y,z)\,+\,\caR(x,z,y)=x\,+\,\caD(x,y,z)$.   For double boundary cuff pants, the main gap lengths are added to the complementary interval length.  The result is 
$\caR(\ell_{\beta},\ell_{\alpha},\ell_{\lambda})$.  
\end{proof}

\hspace{1in}

\noindent\hypertarget{lec3}{{\bfseries Lecture 3: The covolume formula and recursion.}}

The main step is application of the length identity to reduce the action of the mapping class group to an action of smaller mapping class groups, and consequently express the volume as an integral over a length level set.  The result is an integral of products of lower dimensional volume functions - the recursion. 

The approach is illustrated by computing the genus one, one boundary, volume.  The length identity is 
\[
L\,=\,\sum_{\alpha\ \tiny{simple}}\caD(L,\lla,\lla).
\]
Introduce $\stab(\alpha)\subset\mcg$, the stabilizer for $\mcg$ acting on free homotopy classes.  A torus is an elliptic curve with universal cover $\mathbb C$ with involution $z\rightarrow -z$ stabilizing the deck transformation lattice. The involution acts on tori and tori with one puncture or boundary. The involution reverses orientation for the free homotopy class of each simple closed geodesic and the stabilizer $\stab(\alpha)$ is the semi direct product of the Dehn twists by the involution $\mathbb Z/2\mathbb Z$ subgroup.  The involution acts trivially on Teichm\"{u}ller space.  (The torus is one of the exceptional cases where the $\mcg$ action on $\caT$ is not effective.  We will also discuss the torus case below, where a multiplicity is involved.)  A Dehn twist acts on the Teichm\"{u}ller space in FN coordinates by $(\ell,\tau)\rightarrow(\ell,\tau+\ell)$.  The sector $\{0\le\tau<\ell\}$ is a fundamental domain for the  $\stab(\alpha)$ action.  A mapping class $h\in\mcg$ acts on a geodesic-length function by $\lla\circ h^{-1}=\ell_{h(\alpha)}$. 

Write the length identity as
\[
L\,=\,\sum_{\alpha}\,\caD(L,\lla,\lla)\,=\,\sum_{h\in\mcg/\stab(\alpha)}\caD(L,\ell_{h(\alpha)},\ell_{h(\alpha)}),
\]
use the $\mcg$ action on geodesic-length functions, to find
\[
LV(L)\,=\,\int_{\caT(L)/\mcg}\sum_{\mcg/\stab(\alpha)}\caD(L,\lla\circ h^{-1},\lla\circ h^{-1})\,\omega,
\]
change variables on $\caT$ by $p=h(q)$ to find
\[
\sum_{h\in\mcg/\stab(\alpha)}\int_{h(\caT(L)/\mcg)}\caD(L,\lla,\lla)\,d\tau d\ell\,=\,\int_{\caT(L)/\stab(\alpha)}\caD(L,\lla,\lla)\,d\tau d\ell,
\]
and use the $\stab(\alpha)$ fundamental domain, to obtain the integral
\[
\int_0^{\infty}\int_0^{\ell}\,\caD(L,\ell,\ell)\,d\tau d \ell.
\]
The integral in $\tau$ gives a factor of $\ell$.

The derivatives $\partial\caD(x,y,z)/\partial x$ and $\partial\caR(x,y,z)/\partial x$ are simpler than the original functions $\caD$ and $\caR$ -  apply this observation and differentiate in $L$ to obtain a formula for the derivative of $LV(L)$,
\[
\frac{\partial}{\partial L}LV(L)\,=\,\int_0^{\infty}\frac{1}{1+e^{\ell+\frac{L}{2}}}\,+\,\frac{1}{1+e^{\ell-\frac{L}{2}}}\,\ell d \ell\,=\,\frac{\pi^2}{6}\,+\,\frac{L^2}{8}.
\]
The formula $V(L)=\frac{\pi^2}{6}+\frac{L^2}{24}$ results. 

We prepare for the general case.  In algebraic geometry intersection theory, the elliptic involution gives rise to multiplying elliptic intersection counts by a factor of $1/2$. The factor corresponds to the generic fiber of the universal elliptic curve being the quotient of the elliptic curve by its involution.  Along this line, the general volume recursion is simplified if $V_{1,1}(L)$ is formally defined to be $1/2$ of the given value $V(L)$.  In mapping class group theory, the elliptic involution appears as the {\em half Dehn twist} for simple closed curves bounding a torus.  In particular, consider the fundamental group $\pi_1(R)$ of a surface, with the standard presentation $a_1b_1a_1^{-1}b_1^{-1}\cdots a_gb_ga_g^{-1}b_g^{-1}c_1\cdots c_n=1$, with $c_j$ a loop about the $j^{th}$ boundary.  The half Dehn twist about the curve $a_1b_1a_1^{-1}b_1^{-1}$ is the automorphism of $\pi_1(R)$ given by: $a_1\rightarrow b_1a_1^{-1}b_1^{-1}, b_1\rightarrow b_1^{-1}; a_j\rightarrow a_1^{-1}a_ja_1, b_j\rightarrow a_1^{-1}b_ja_1,\,\mbox{for }j=2\dots g, \mbox{ and } c_j\rightarrow a_1^{-1}c_ja_1,\,\mbox{for }j=1\dots n$.  The square of a half Dehn twist is a Dehn twist and a half Dehn twist acts on the associated FN parameters by $(\ell,\tau)\rightarrow(\ell,\tau+\ell/2)$.  

We now set up for the covolume formula.  Let $R$ be a hyperbolic surface with geodesic boundaries 
$\beta_1,\,\dots,\beta_n$.  Consider a weighted multicurve
\[
\gamma\,=\,\sum_{j=1}^m a_j\gamma_j,
\]
where $a_j$ are real weights and $\gamma_j$ are distinct, disjoint, simple closed geodesics.  Define $\stab(\gamma)\subset\mcg$ to be the mapping classes stabilizing the collection of unlabeled, weighted geodesics - elements of $\stab(\gamma)$ may permute components of the multicurve with equal weights. Write $\stab(\gamma_j)$ for the stabilizer of an individual geodesic and $\stab_0(\gamma_j)$ for the subgroup of elements preserving orientation.  

Write $R(\gamma)$ for the surface cut open along the  $\gamma$ - each $\gamma_j$ gives rise to two new boundaries - $R(\gamma)$ may be disconnected.  Write $\caT(R(\gamma);\mathbf x),\,\mathbf x=(x_1,\dots,x_m)$ for the (product) Teichm\"{u}ller space of the cut open surface with the pair of boundaries for $\gamma_j$ having length $x_j$. Denote by $\mcg(R(\gamma))$ the product of mapping class groups of the components of $R(\gamma)$ and by $\caT(R(\gamma);\mathbf x)/\mcg(R(\gamma))$ the corresponding product of moduli spaces. For the product of symplectic forms on $\caT(R(\gamma);\mathbf x)$ corresponding to the components of $R(\gamma)$, the volume $V(R(\gamma);\mathbf x)$ is the product of volumes of the component moduli spaces, where again the pair of boundaries for $\gamma_j$ have common length $x_j$.  Considerations  also involve the finite symmetry group
\[
\sym(\gamma)\,=\,\stab(\gamma)/\cap_j\stab_0(\gamma_j)
\]
of mapping classes that possibly permute and reverse orientation of the $\gamma$ elements.    

Summing the translations of a function over a group gives a group action invariant function.  Begin with a function $f$, suitably small at infinity, and introduce the $\mcg$ sum
\begin{equation}
f_{\gamma}(R)=\sum_{\mcg/\stab(\gamma)}f\big(\sum_{j=1}^ma_j\ell_{h(\gamma_j)}(R)\big).
\end{equation}
The next theorem expresses the moduli space integral 
\[
\int_{\caM(R)}\,f_{\gamma}\,dV
\]
as a weighted integral of lower dimensional moduli space volumes.

\begin{thrm*}\cite[Thrm. 7.1]{Mirvol} \textup{The covolume formula.}\label{covol} For a weighted  $\gamma=\sum_{j=1}^ma_j\gamma_j$ and the $\mcg$ sum of a function $f$, small at infinity, then
\[
\int_{\caT(R)/\mcg}f_{\gamma}\,dV=(|\mbox{Sym}(\gamma)|)^{-1}\int_{\mathbb R_{>0}^m}f(|\mathbf x|)V(R(\gamma);\mathbf x)\,\mathbf x\cdot d\mathbf x
\]
where $|\mathbf x|=\sum_ja_jx_j$ and $\mathbf x\cdot d\mathbf x=x_1\cdots x_mdx_1\cdots dx_m$.
\end{thrm*}
\begin{proof} Corresponding to the components $R'$ of the cut open surface $R(\gamma)$,  consider the short exact sequence for mapping class groups, 
\[
1\longrightarrow\prod_j\mbox{Dehn}(\gamma_j) \longrightarrow \bigcap_j\stab_0(\gamma_j)\longrightarrow \prod_{R(\gamma)\ components} \mcg(R')\longrightarrow 1,
\]
and the associated fibration of Teichm\"{u}ller spaces from Fenchel-Nielsen coordinates,
\begin{equation*}
\xymatrix{ \prod_{R(\gamma)\ components}\caT(R')\  \ar@{^{(}->}[r] & \caT(R) \ar[d] \\ & \prod_{\gamma_j}\mathbb R_{>0}\times \mathbb R\,.}
\end{equation*}
(The short exact sequence places half Dehn twists in the mapping class groups of the tori with single boundaries.)  The $d\ell\wedge d\tau$ formula provides that the fibration is a fibration of symplectic manifolds.  

To establish the  formula, first write for coset sums
\[
\sum\limits_{\mcg/\cap_j\stab_0(\gamma_j)}f\,=\,\sum\limits_{\mcg/\stab(\gamma)}\ \sum\limits_{\stab(\gamma)/\cap_j\stab_0(\gamma_j)}f\,=\,|\sym(\gamma)|\,f_{\gamma},
\] 
using that $f_{\gamma}$ is $\sym(\gamma)$ invariant for the second equality.  Substitute the resulting formula for $f_{\gamma}$ into the integral, and unfold the sum (express the $\mcg/\cap_j\stab_0(\gamma_j)$ translation sum as a sum of translates of a $\mcg$ fundamental domain) to obtain the equality
\[
\int_{\caT(R)/\mcg}f_{\gamma}\,dV=(|\sym(\gamma)|)^{-1}\int_{\caT(R)/\cap_j\stab_0(\gamma_j)} f\,dV.  
\]
Substitute the fibration
\begin{equation*}
\xymatrix{ \prod_{R(\gamma)\ components}\caT(R')/\mcg(R')\  \ar@{^{(}->}[r] & \caT(R)/\bigcap_j\stab_0(\gamma_j) \ar[d] \\ & \prod_{\gamma_j}(\mathbb R_{>0}\times\mathbb R)/\mbox{Dehn}_*(\gamma_j)\,,}
\end{equation*}
where $\mbox{Dehn}_*(\gamma_j)$ is generated by a half twist if the curve bounds a torus with a single boundary and otherwise is generated by a simple twist. 
Substitute the factorization of the volume element 
\[
dV=\prod_{R(\gamma)\ components}dV(R')\times\prod_{\gamma_j} d\ell_j\wedge d\tau_j.
\]
The function $f$ depends only on the values $\mathbf x$.  For the values $\mathbf x$ fixed, perform the $\prod\caT(R')/\mcg(R')$ integration to obtain the product volume $V(R(\gamma);\mathbf x)$.  Finally $\mbox{Dehn}_*(\gamma_j)$ acts only on the variable $\tau_j$ with fundamental domain $0<\tau_j<\ell_j/2$ if $\gamma_j$ bounds a torus with a single boundary or otherwise with fundamental domain $0<\tau_j<\ell_j$.  For a torus with a single boundary, the action is accounted for by using the volume value that is $1/2$ the original $V(L)$.  The right hand side of the formula is now established.  
\end{proof}

We are now ready to apply the covolume formula to the length identity.  The application follows the genus one example.  Again - it is essential to use the $\mcg$ geodesic-length function action $\lla\circ h^{-1}=\ell_{h(\alpha)}$, to consider a sum over topological configurations as a sum of $\mcg$ translates of a function. Consider the action on configurations.  The mapping class group naturally acts on $\caB$, the set of non trivial free homotopy classes of simple curves with endpoints on the boundary, rel the boundary.  Recall the correspondences: elements of $\caB$ $\Longleftrightarrow$ wire frames $\Longleftrightarrow$ boundary pants 
$\caP$. Viewing Figure \ref{fig:recursion}, the $\mcg$ orbits on $\caB$ are of three types, describing location of the boundary pants $\caP$.

\begin{itemize}
  \item Orbits for simple curves from $\beta_1$ to $\beta_1$.
  \begin{itemize}  
	\item A single orbit for $R-\caP$ connected, with $\sym(\beta_1,\alpha_1,\alpha_2)=\mathbb Z/2\mathbb Z$. 
	\item A collection of orbits for $R-\caP$ disconnected.  The orbits are classified by the joint partitions of genus $g=g_1+g_2$ and of labeled boundaries $\{\beta_2,\dots,\beta_n\}$, with each resulting component with negative Euler characteristic.  In general $\sym(\beta_1,\alpha_1,\alpha_2)=1$, except in the special case of one boundary and $g_1=g_2$. 
   \end{itemize}
   \item A collection of orbits, one for each choice of a second boundary.   In particular, an orbit for simple curves from $\beta_1$ to $\beta_j,\,j\ne 1$.  A resulting surface $R-\caP$ is connected, with $\sym(\beta_1,\beta_j,\alpha)=1$.  
\end{itemize}

\noindent Consider Theorem \ref{lenid}, and integrate each side of the identity
\[
L_1\,=\,\sum_{\alpha_1,\alpha_2}\,\caD(L_1,\ell_{\alpha_1}(R),\ell_{\alpha_2}(R))\,+\,
\sum_{j=2}^n\sum_{\alpha}\,\caR(L_1,L_j,\ell_{\alpha}(R)),
\]
over the moduli space of $R$ relative to the volume $dV$.  Form the $L_1$ partial derivative of each side to simplify the quantities $\caD$ and $\caR$. Apply formulas (\ref{drdef}) for the right hand side.  Express the right hand side as individual sums for given orbit types.  Apply Theorem \ref{covol} for each orbit type to find integrals in terms of lower dimensional moduli volumes as follows.
\begin{itemize}
  \item For the sum over simple curves from $\beta_1$ to $\beta_1$ with $R-\caP$ connected, the summand function is $\caD$ with
\[
\frac{\partial \caD}{\partial L_1}\,=\,H(\ell_{\alpha_1}+\ell_{\alpha_2},L_1),
\]
a function of length of a multicurve, and the resulting integral is
\[
\int_0^{\infty}\int_0^{\infty}H(x+y,L_1)V_{g-1}(x,y,\widehat L)\,xy\,dxdy,
\]
for $\widehat L=(L_2,\dots,L_n)$.
  \item For the sum over simple curves from $\beta_1$ to $\beta_1$ with $R-\caP$ disconnected, the summand function is $\caD$ with
\[
\frac{\partial \caD}{\partial L_1}\,=\,H(\ell_{\alpha_1}+\ell_{\alpha_2},L_1),
\]
a function of length of a multicurve, and the resulting integral is
\[
\int_0^{\infty}\int_0^{\infty}\sum_{\stackrel{g_1+g_2=g}{I_1\amalg I_2=\{2,\dots,n\}}} H(x+y,L_1)V_{g_1}(x,L_{I_1})V_{g_2}(y,L_{I_2})\,xy\,dxdy.
\]
   \item For the sum over simple curves from $\beta_1$ to $\beta_j,\,j\ne 1$, the summand function is $\caR$ with
\[
\frac{\partial\caR}{\partial L_1}\,=\,\frac12H(\lla,L_1+L_j)\,+\,\frac12H(\lla,L_1-L_j),
\]
a sum of functions of weighted length of a multicurve, and the resulting integral is
\[
\int_0^{\infty}\frac12\sum_{j=2}^n\big(H(x,L_1+L_j)+H(x,L_1-L_j)\big)V_g(x,L_2,\dots,\widehat{L_j},\dots,L_n)\,x\,dx,
\]
where $L_j$ is omitted from the argument list of $V_g$. 
\end{itemize}
Compare to the \hyperlink{volrecur}{end of Lecture 1} - the volume recursion is established.  The volume 
{\em function} $V_{g,n}(L)$ is recursively determined. 

What type of function is $V_{g,n}(L)$?  The recursion involves two elementary integrals, see formula (\ref{hdef}) above for the definition of $H$,
\[
\int_0^{\infty} x^{2j+1}H(x,t)\,dx\qquad\mbox{and}\qquad
\int_0^{\infty}\int_0^{\infty}x^{2j+1}y^{2k+1}H(x+y,t)\,dxdy.
\]
By direct calculation, see \cite[formula (6.2) and Lemma 6.2]{Mirvol}, each integral is a polynomial in $t^2$ with each coefficient a product of factorials and the Riemann zeta function at a non negative even integer, each coefficient is a positive rational multiple of an appropriate power of $\pi$.   The first polynomial has degree $j+1$ in $t^2$, while the second has degree $i+j+2$.  The first part of Theorem \ref{first} is now established.  The second part is the subject of the next lecture.

\noindent\hypertarget{lec4}{{\bfseries Lecture 4: Symplectic reduction, principal $S^1$ bundles and the normal form.}}

The goal of the lecture is to establish the following theorem.   The formula combines with Theorem \ref{wpkap}, the WP kappa equation $\omega=2\pi^2\kappa_1$, to provide that the coefficients of the volume polynomial $V_g(L)$ are tautological intersection numbers.  The result completes the proof of Theorem \ref{first}.
\begin{thrm*}\label{norfor} For $d=\frac12\dim_{\mathbb R}\TgL=\dim_{\mathbb C}\Tgn$, then
\[
V_g(L)\ =\ \frac{1}{d!}\int_{\TgL/\mcg}\omTL^d\ =\ 
\frac{1}{d!}\int_{\Tgn/\mcg} \big(\omega\,+\,\sum_{j=1}^n\frac{L_j^2}{2}\psi_j\big)^d .
\]
\end{thrm*}
\noindent The proof is essentially by establishing a cohomology equivalence between symplectic spaces - combining symplectic reduction, the Duistermaat-Heckman theorem and explicit geometry to obtain the formula.

The considerations of the lecture are presented for the appropriate Teichm\"{u}ller spaces $\caT$ and the open moduli spaces $\caT/\mcg$.  The constructions are compatible with the augmentation construction.  The results are valid for the appropriate Teichm\"{u}ller spaces and compactified moduli spaces.  The compactified moduli spaces are orbifolds.  The following results are for cohomology statements over $\mathbb Q$; cohomology arguments over $\mathbb Q$ for manifolds are in general also valid for orbifolds. Alternatively, the orbifold matter can be bypassed by applying the general result that compactified moduli spaces have manifold finite covers, \cite{BogPik, Looijsmo}.  

\noindent {\bfseries The Teichm\"{u}ller spaces.}  We consider the trio.
\begin{itemize}
  \item $\TgL$ - the space of marked genus $g$ hyperbolic surfaces, with geodesic boundaries $\beta_1,\dots,\beta_n$ of prescribed lengths $L_1,\dots,,L_n$.  A hyperbolic surface can be doubled across its geodesic boundary to obtain a compact hyperbolic surface of higher genus.  Accordingly, $\TgL$ can be considered as a locus in $\caT_{2g+n-1}$.  The symplectic form $\omega$ of the image Teichm\"{u}ller space restricts to the locus and defines a symplectic form on the locus.  A pants decomposition for a surface with boundary, can be doubled to give a pants decomposition for a doubled surface.  The doubled pants decomposition is characterized by containing the geodesics $\beta_1,\dots,\beta_n$ and being symmetric.  Fenchel-Nielsen coordinates and the $d\ell\wedge d\tau$ formula are applied for doubled decompositions to obtain a description of the locus $\TgL\subset\caT_{2g+n-1}$, and to define a symplectic from $\omTL$.  The symplectic form  is given as $\sum_j d\ell_j\wedge d\tau_j$ (without boundary parameters) for any pants decomposition of a surface with boundaries. $\mcg$ invariance is immediate. 
     
  \item $\That$ - the space of marked genus $g$ hyperbolic surfaces, with pointed geodesic boundaries $\beta_1,\dots,\beta_n$ - boundary lengths are allowed to vary {\em and} a variable point is given on each boundary.  The $\mathbb R$ dimension of $\That$ is $2n$ greater than the $\mathbb R$ dimension of $\TgL$. Symplectic reduction requires a symplectic form on $\That$, that is equivalent to $\omTL$ on $L$ level sets and is invariant under rotating points on boundaries.  A form is given by describing $\That$ as a higher dimensional Teichm\"{u}ller space.  
 
\begin{figure}[htbp] 
  \centering
  \includegraphics[bb=0 0 437 127,width=4.5in,height=1.31in,keepaspectratio]{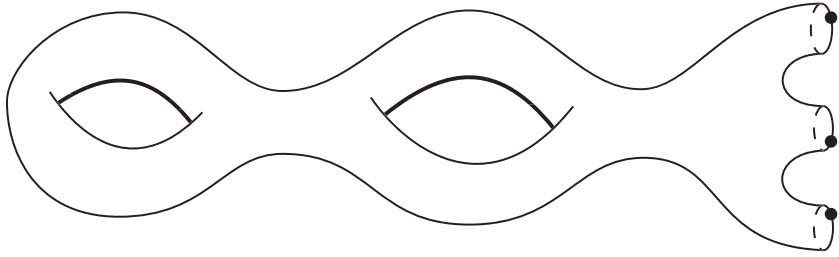}
  \caption{A genus $2$ surface with three pointed boundaries.}
  \label{fig:genus23h}
\end{figure}

To this purpose, introduce almost tight pants - pairs of pants with two labeled boundaries being punctures (length zero) and a third boundary of prescribed length.  An almost tight pants will be glued to each surface boundary $\beta_j$.  The pants equatorial reflection defines symmetric points on the pants boundary; the puncture labeling uniquely determines an equatorial boundary point.  A standard model for a surface with pointed geodesic boundaries is given by gluing on almost tight pants - at each boundary glue on an almost tight pants with matching boundary length and the equatorial point aligned with the point on the boundary.  The construction does not involve choices, so is natural with respect to marking homeomorphisms and the $\mcg$ action.  
\begin{figure}[htbp] 
  \centering
  \includegraphics[bb=0 0 461 265,width=2in,height=1.15in,keepaspectratio]{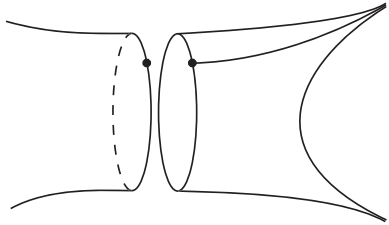}
  \caption{Aligning boundary and equatorial points to glue on almost tight pants.}
  \label{fig:almtight}
\end{figure}

For a surface $R$ with labeled, pointed boundaries, write $\widehat R$ for the standard model surface with glued on almost tight pants.  The punctures of $\widehat R$ are labeled in pairs.  By the general hyperbolic collar result, small length geodesics are necessarily disjoint
 \cite{Busbook,Wlcbms}.  For the lengths $L_1,\dots,L_n$ suitably small, the labeled geodesics $\beta_1,\dots,\beta_n$ are uniquely determined on the surface $\widehat R$ by having small length and bounding labeled punctures.  The pointed boundary, marked surface $R$ {\em is equivalent} to the marked surface $\widehat R$ modulo Dehn twists about the $\beta_j$ (Dehn twists, since the boundary points are given on a circle).   In particular, for $c$ suitably small, the equivalence is between the open subset $\{L<c\}$ of $\That$ and the open subset $\{L<c\}$ in $\caT_{g,2n}/\prod_j\mbox{Dehn}(\beta_j)$.

Definition and equivalence of $S^1$ principal bundles are next.  Considerations begin with the short exact sequence from Theorem \ref{covol},
\[
1\longrightarrow\mbox{Dehn}(\beta)=\prod_j\mbox{Dehn}(\beta_j) \longrightarrow\mbox{Stab}(\beta)= \bigcap_j\stab(\beta_j)\longrightarrow  \mcg(R)\longrightarrow 1,
\]
(now $\stab(\beta_j)=\stab_0(\beta_j)$, since an orientation preserving pants homeomorphism preserves boundary orientation).  The geodesics $\beta_1,\dots,\beta_n$ define subsets of the Riemann surface bundles (the universal curves) over $\That$ and over $\caT_{g,2n}/\mbox{Dehn}(\beta)$.  The subsets define oriented circle bundles, provided automorphisms of the Riemann surfaces act at most as rotations on the individual geodesics.  The small lengths $L_1,\dots,L_n$ and labeled boundaries provide the condition.  The geodesics define circle bundles over Teichm\"{u}ller bases.  We see below that rotation along geodesics defines an $S^1$ principal structure.  Next, from the above short exact sequence and the definition of marking - the equivalence between geodesics $\beta\subset R$ and $\beta\subset\widehat R$, and the projections of circle bundles to bases, commute with the actions of $\mcg(R)\approx \stab(\beta)/\mbox{Dehn}(\beta)$.  The geodesics define equivalent orbifold $S^1$ principal bundles over the quotients $\{L<c\}\subset\That/\mcg(R)\times(S^1)^n$ (the $\mcg(R)$ and $(S^1)^n$ actions on $\That$ commute) and $\{L<c\}\subset\caT_{g,2n}/\stab(\beta)$.  

The symplectic form of $\caT_{g,2n}$ defines a symplectic form $\omThat$ on the open subset $\{L<c\}$.  Fenchel-Nielsen coordinates and the $d\ell\wedge d\tau$ formula are applied.  The form $\omThat$ is given as $\sum_kd\ell_k\wedge d\tau_k$ for any pants decomposition of $\widehat R$ containing the multicurve $\beta$.  Importantly, the form $\omThat$ is given for surfaces $R$ with pointed boundaries $\beta_j$, by an extended interpretation of the $d\ell\wedge d\tau$ formula, with a term for each boundary, now with the interpretation that $\tau(\beta_j)$ parameterizes the location of the specified point.  See Figure \ref{fig:genus3t}, the parameter $\tau(\beta_j)$ increasing corresponds to the point moving on the boundary with the surface interior on the right.  $\mcg(R)$ invariance of the symplectic form is immediate.  Restriction of the form to $L$ level sets and invariance under rotating boundary points are discussed below.

  \item $\Tgn$ - the space of marked genus $g$ hyperbolic surfaces with $n$ punctures.  $\Tgn$ has the $\mcg$ invariant symplectic form $\omega$. $\Tgn$ coincides with the Teichm\"{u}ller space $\Tg(0)$, where surface boundary lengths are zero. 
\end{itemize}
We will relate the three symplectic manifolds. 

\noindent{\bfseries Symplectic reduction for $\That$.}  We consider the Hamiltonian geometry of FN twists, geodesic-lengths and especially the moment map
\[
\That\ \stackrel{\mu}{\longrightarrow}\ \widehat L=(L^2_1/2,\dots,L^2_n/2)\in \mathbb R_{\ge 0}^n.
\]
Write $t_j$ for the unit speed infinitesimal rotation of the point on the boundary $\beta_j$; $t_j$ is a vector field on $\That$.  In terms of the standard model surfaces, $t_j$ is the FN infinitesimal twist vector field for $\beta_j$, and $t_j$ is the infinitesimal rotation of the $j^{th}$ almost tight pants.  By twist-length duality, we have $\omThat(-t_j,\ )=dL_j$ and the scaled $-L_jt_j$ is unit infinitesimal rotation (unit time flow is a full rotation).  The function $\frac12L_j^2$ is the corresponding Hamiltonian potential, since $\omThat(-L_jt_j,\ )=d(\frac12L_j^2)$ (the momentum $\frac12L^2$ determines the twist sign/orientation).  The vector fields $-L_jt_j$ are the infinitesimal generators for the $(S^1)^n$ action on $\That$, given by rotating the boundary points.  The symplectic form $\omThat$ is twist invariant and we are ready for symplectic reduction, ready to consider the quotient $\That/(S^1)^n$.

\begin{figure}[htbp] 
  \centering
  \includegraphics[bb=0 0 454 144,width=5in,height=1.59in,keepaspectratio]{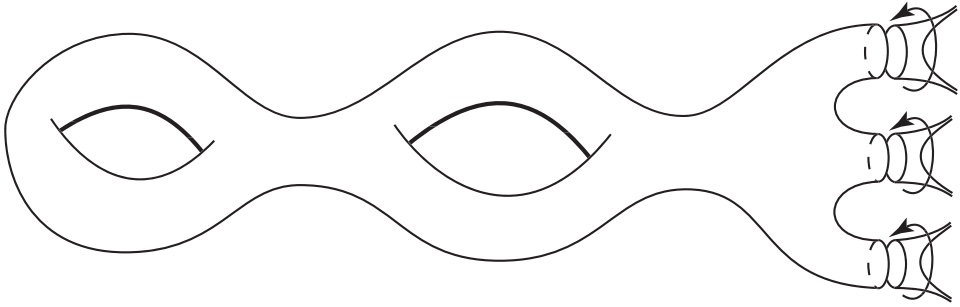}
  \caption{Positive rotations for an $(S^1)^3$ action.}
  \label{fig:S1act}
\end{figure}

A level set of the moment map $\mu:\That\longrightarrow\mathbb R^n$ is a locus of prescribed 
$\beta$ length hyperbolic surfaces.  The group $(S^1)^n$ acts on level sets by rotating almost tight pants.  The quotient of a level set by the group is naturally $\TgL$ - the level set prescribes the boundary lengths and the group action removes the location information for the points.

\begin{prop}\textup{Symplectic reduction.}\label{symred} For $\That/(S^1)^n\approx \TgL$, then 
\[
\omThat \big|_{\mu^{-1}(\widehat L)}/(S^1)^n\ \approx\ \omTL.
\]
\end{prop}
\begin{proof}
The form $\omThat$ is given by the $d\ell\wedge d\tau$ formula for any pants decomposition of a standard model surface containing the multicurve $\beta$.  The differentials $dL_j$ vanish on $\mu$ level sets and the formula reduces to the sum for a pants decomposition of a surface with boundary, a sum without boundary parameters - the $\omTL$ formula.
\end{proof}
   
\noindent{\bfseries $S^1$ principal bundles.}  We review basics about characteristic classes.

\begin{defn} Let $\pi:P\longrightarrow M$ be a smooth circle bundle over a smooth compact manifold $M$.  The bundle is $S^1$ principal provided,
\begin{enumerate}
  \item $S^1$ acts freely on $P$,
  \item $\pi(p_1)=\pi(p_2)$ if and only if there exists $s\in S^1$, such that $p_1\cdot s =p_2$.
\end{enumerate}
A connection for an $S^1$ principal bundle is a smooth distribution $\caH\subset\mathbf T P$ of tangent subspaces such that,
\begin{enumerate}
  \item $\mathbf T_pP=\caH_p\oplus\ker \pi_*\big|_p$, for each $p\in P$, 
  \item $s^*\caH_p=\caH_{p\cdot s}$.  
\end{enumerate}
\end{defn} 

\noindent A connection is uniquely given as $\caH=\ker A$, for a $1$-form $A$ on $P$, provided $A$ is $S^1$ invariant and $A(\dot{s})=1$. An $S^1$ invariant inner product $\langle\ ,\ \rangle$ provides an example of an invariant $1$-form by $A(v)=\langle v,\dot{s}\rangle/\langle\dot{s},\dot{s}\rangle$. The curvature $2$-form on $P$ for a connection is $\Phi(v,w)=dA(\operatorname{hor}v, \operatorname{hor}w)$, for $\operatorname{hor}$ the {\em horizontal} projection of $\mathbf TP$ to $\caH$.
\begin{thrm*} \cite{MilSta}. 
There exists a unique closed $2$-form $\Omega$ on $M$, such that $\Phi=\pi^*\Omega$.  The cohomology class of $\Omega$ is independent of the choice of $S^1$ principal connection for $P$ and the first Chern class is $c_1(P)\,=\,[\Omega]\in H^2(M,\mathbb Z)$.
\end{thrm*}

As above, the variable point on the boundary $\beta_j$ of the surface $R$ defines an $S^1$ principal bundle $\widehat \beta_j$ over $\TgL$; $S^1$ acts by moving the point with the surface interior on the left.  A choice of connection for the bundle gives a first Chern class $c_1(\widehat\beta_j)$. 

\noindent {\bfseries Applying the Duistermaat-Heckman theorem.}   We extend the definition of $\That$ to include $L=0$; geodesic boundaries of $R$ can be replaced with punctures. Hyperbolic structures converge for boundary lengths tending to zero; in particular collar regions converge to cusp regions.  The extension of $\That$ is given by parameterizing boundary points by points on a collar/cusp region boundary. 

We recall basics about collars and cusps. For a geodesic $\alpha$ of length $\ell_{\alpha}$, the standard collar in the upper half plane $\mathbb H$ is $\caC(\ell_{\alpha})=\{d(z,i\mathbb R_+)\le w(\alpha)\}$, for the half width $w(\alpha)$ given by  $\sinh w(\alpha)\sinh\ell_{\alpha}/2=1$.  The quotient cylinder $\{d(z,i\mathbb R_+)\le w(\alpha)\}/\langle z\mapsto e^{\ell_{\alpha}}z\rangle$ embeds into $R$ to give a collar neighborhood of the geodesic.  For a cusp, the standard cusp in $\mathbb H$ is $\caC_{\infty}=\{\Im z\ge1/2\}$.  The quotient cylinder $\{\Im z\ge1/2\}/\langle z\mapsto z+1\rangle$ embeds into $R$ to give a cusp region. The boundary of a collar, for $\ell_{\alpha}$ bounded, and boundary of a cusp region have length approximately 2. Collars and cusp regions are foliated by geodesics normal to the boundary.  For geodesic-lengths tending to zero, half collar neighborhoods Gromov-Hausdorff converge to a cusp region (convergence is uniform on bounded distance neighborhoods of the boundary); boundaries and geodesics normal to the boundary converge.  
\begin{figure}[htbp] 
  \centering
  \includegraphics[bb=0 0 509 158,width=5in,height=1.55in,keepaspectratio]{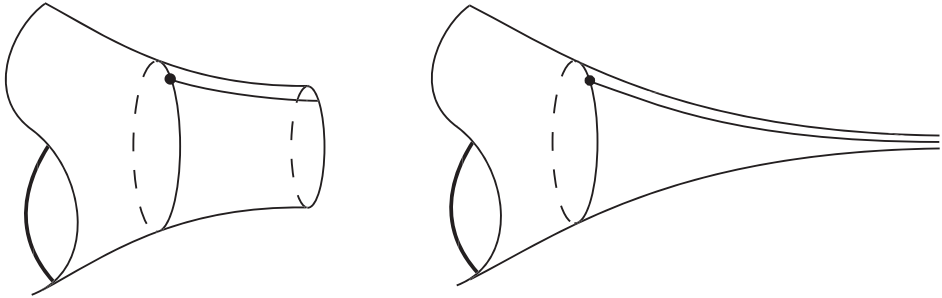}
  \caption{Projecting along geodesics to a collar and a cusp region boundary.}
  \label{fig:projbdry}
\end{figure}

The geodesics normal to the boundary of a collar provide a projection from the core geodesic to each collar boundary.  The projection is used to note that prescribing a point on a geodesic boundary of $R$ is equivalent to prescribing a point on the boundary of the half collar neighborhood of the geodesic.  Since collars and their boundaries converge to a cusp region and its boundary, for core geodesic length tending to zero, we have a description for the extension of the definition of $\That$ to include surfaces with collections of lengths $L_j$ zero.   

The standard cusp region is uniformized by the variable $w=e^{2\pi iz}$.  A point on the cusp region boundary $\Im z=1/2$ corresponds to a point on $|w|=e^{-\pi}$ and given the factor $e^{-\pi}$, a point on the circle corresponds to a tangent vector at the origin.  The variable $w$ is unique modulo multiplication by a unimodular number; the identification of the circle with tangent vectors at the origin is canonical.  For an $S^1$ infinitesimal generator $-L_jt_j$, displacement is to the left when crossing the geodesic, (compare to the Figure \ref{fig:genus3t} positive twist, right displacement) the reference point moves with the surface interior on its left, the tangent vector at the origin rotates clockwise, and a dual cotangent vector rotates counter clockwise (the positive direction for a $\mathbb C$-line). Combining equivalences, the $S^1$ principal bundle of a point on a cusp region boundary is equivalent to a non zero vector in the cotangent line for the puncture.

\noindent {\bfseries Proof of Theorem \ref{norfor}.}  The $\mcg(R)$ and $(S^1)^n$ actions on $\That$ commute; consider the quotient $\That/\mcg(R)\times(S^1)^n$.  By the Duistermaat-Heckman theorem, \cite[Chapter 30, Theorem 30.8]{Cansymp}, for small values of $\widehat L$, including $0$, the reduced level sets $\mu^{-1}(\widehat L)/(S^1)^n$ are mutually diffeomorphic. Furthermore by Duistermaat-Heckman, the $L$ level set reduced symplectic form  $\omThat \big|_{\mu^{-1}(\widehat L)}/(S^1)^n$ is cohomologous to the sum of the $0$ level set reduced form $\omThat \big|_{\mu^{-1}( 0)}/(S^1)^n$ 
and the contributions $(L_j^2/2)\,c_1(\widehat\beta_j)$, for $c_1(\widehat\beta_j)$ the first Chern class for the $S^1$ principal bundle of a point varying on the $j^{th}$ cusp region boundary. Combining with Proposition \ref{symred}, gives the desired cohomology equivalence, 
\[
\omTL\ \equiv\ \omega_{\Tg(0)}\ +\ \sum_{j=1}^n\frac{L_j^2}{2}c_1(\widehat\beta_j).
\]
By the description of collars and cotangent lines at punctures, the circle bundle $\widehat\beta_j$ is topologically equivalent to the psi line bundle $\psi_j$ (see Lecture 1) with equality of first Chern classes. The proof is finished.

\noindent\hypertarget{lec5}{{\bfseries Lecture 5: The pattern of intersection numbers and Witten-Kontsevich.}}

We begin with the discussion of Harris-Morrison \cite[pgs. 71-75]{HMbook}.  For a finite sequence of non negative integers $\{\alpha_j\}$, define the top $\psi$-intersection number by
\[
\langle\tau_{\alpha_1}\tau_{\alpha_2}\cdots\tau_{\alpha_n}\rangle_g\,=\,\int_{\overline{\mathcal M}_{g,n}}\psi_1^{\alpha_1}\psi_2^{\alpha_2}\cdots\psi_n^{\alpha_n}.
\]
For a non trivial pairing,  the genus $g$, number of punctures $n$, and exponents $\alpha_j$ are related by $3g-3+n=\sum_{j=1}^n \alpha_j$, otherwise the pairing is defined as zero.  More generally using exponents to denote powers (repetitions) of the variables $\tau$, define
\[
\langle\tau_0^{d_0}\tau_1^{d_1}\cdots\tau_m^{d_m}\rangle_g\,=\,\int_{\overline{\mathcal M}_{g,d}}\prod_{j=0}^m\,\prod_{k=1}^{d_j}\,\psi_{(j,k)}^{\,j}.
\]
($\tau_j^{d_j}$ denotes that for $d_j$ punctures, the associated $\psi$ is raised to the $j^{th}$ power; the subscripts $(j,k)$ are distinct puncture labels.)   For the second pairing, the formal count of punctures is $d=\sum_{j=0}^md_j$, and the formal degree of the product is the count of $\psi$ factors $\sum_{j=0}^mjd_j$.   For a non trivial pairing, the genus, number of punctures and degree are related by $3g-3+d=\sum_{j=0}^mjd_j$, otherwise the pairing is zero.  The psi classes are known to be positive - integrals of products over subvarieties are positive; the non trivial pairings are positive, matching Mirzakhani's positivity of volume polynomial coefficients, see Theorem \ref{first}.

Witten considered a partition function (probability of states), for two-dimensional gravity.  For an infinite vector $\mathbf t=(t_0,t_1,\dots,t_n,\dots)$, and $\gamma$ the formal sum $\gamma=\sum_{j=0}^{\infty}t_j\tau_j$, Witten introduced a genus $g$ generating function for $\tau$ products,
\[
F_g(\mathbf t)=\sum_{n=0}^{\infty}\frac{\langle\gamma^n\rangle_g}{n!},
\]
in which the numerator is defined by monomial expansion, resulting in the formal power series
\[
F_g(\mathbf t)=\sum_{\{d_j\}}\ \langle\prod_{j=0}^{\infty}\tau_j^{d_j}\rangle_g\, \prod_{j=0}^{\infty}\frac{t_j^{d_j}}{d_j!},
\]
where the sum is over all sequences of non negative integers $\{d_j\}$ with only finitely many non zero terms.  --By Theorem \ref{first}, the intersection numbers $\langle\tau_0^{d_0}\tau_1^{d_1}\cdots\tau_m^{d_m}\rangle_g$ are the coefficients of the leading terms of the volume polynomials $V_{g,n}(L)$.--

The quantum gravity partition function is
\[
\mathbf F(\lambda,\mathbf t)\,=\,\sum_{g=0}^{\infty}\lambda^{2g-2}F_g(\mathbf t).
\]
Based on a realization of the function in terms of matrix integrals, Witten conjectured that the partition function should satisfy two forms of the Korteweg-deVries (KdV) equations.  
Kontsevich gave a proof of the conjecture using a cell decomposition of the moduli spaces $\Mgn$, \cite{Kont}.  Cells are enumerated by ribbon graphs/fat graphs.  Kontsevich encoded the intersection numbers in an enumeration of trivalent ribbon graphs.  He then used Feynman diagram techniques and a matrix Airy integral to establish Witten's conjectures.  

Two basic relations for the intersection numbers are: for $n>0$ and $\sum_i\alpha_i=3g-2+n>0$, the
\[
\mbox{string equation}\qquad \langle\tau_0\tau_{\alpha_1}\cdots\tau_{\alpha_n}\rangle_g\,=\,\sum_{\alpha_i\ne 0}\langle\tau_{\alpha_1}\cdots\tau_{\alpha_i-1}\cdots\tau_{\alpha_n}\rangle_g,
\]
and for $n\ge0$ and $\sum_i\alpha_i=3g-3+n\ge0$, the
\[
\mbox{dilaton equation}\qquad\quad \langle\tau_1\tau_{\alpha_1}\cdots\tau_{\alpha_n}\rangle_g\,=\,(2g-2+n)\langle\tau_{\alpha_1}\cdots\tau_{\alpha_n}\rangle_g.
\]
The first equation is for adding a new puncture without an associated factor of $\psi$ in the product, while the second equation is for adding a new puncture with a single associated factor of $\psi$.  

Similar to setting $V_{0,3}(L)=1$, the intersection symbol for the thrice punctured sphere is normalized to $\langle\tau_0^3\rangle_0=1$. The general genus $0$ formula is
\[
\langle \tau_{\alpha_1}\cdots\tau_{\alpha_n}\rangle_0\,=\,\bigg(\frac{n-3}{\alpha_1!\cdots\alpha_n!}\bigg),
\]
with the right hand side a multinomial coefficient for $n-3$.  The genus $0$ string equation is simply Pascal's multinomial neighbor relation.  

For genus $1$, Theorem \ref{norfor} and the WP kappa equation, Theorem \ref{wpkap}, give $V(L)=\int_{\mathcal M_{1,1}} 2\pi^2\kappa_1+\frac{L^2}{2}\psi$.  The formula combines with the Lecture 3 calculation  $V_{1,1}(L)=\frac{\pi^2}{12}+\frac{L^2}{48}$ (now including the elliptic involution $\frac12$ factor) to provide the evaluations, 
\[
\frac12\int_{\mathcal M_{1,1}}\kappa_1\,=\,\langle\tau_1\rangle_1\,=\,\frac{1}{24}.
\]   
General genus $1$ evaluations are found from the single evaluation by applying the string and dilaton equations.  A consequence of the Witten conjecture is that all $\langle\tau\rangle$ intersections can be calculated from the initial values $\langle\tau_0^3\rangle_0=1$ and $\langle\tau_1\rangle_1=\frac{1}{24}$, using the Virasoro equations $L_n(e^{\mathbf F})=0$ for the partition function described below.

In the \hyperlink{volrecur}{volume recursion}, leading coefficients are obtained from leading coefficients - the recursion specializes to leading coefficients, see \cite[Lemma 5.3]{Mirwitt}.  We now sketch the application of the specialized recursion to relations for the partition function and a solution of Witten's conjecture.  

Relations come from the Virasoro Lie algebra.  The Witt subalgebra is generated by the differential operators $\mathcal L_n=-z^{n+1}\partial/\partial z,\, n\ge-1$, with commutators $[\mathcal L_n,\mathcal L_m]=(n-m)\mathcal L_{n+m}$. The string and dilaton equations can be written as linear homogeneous differential equations for the exponential $e^{\mathbf F}$ of the partition function.  The differential operator for the string equation is
\[
L_{-1}\,=\,-\frac{\partial}{\partial t_0}\,+\,\frac{\lambda^{-2}}{2}\,t_0^2\,+\,\sum_{j=0}^{\infty}t_{j+1}\frac{\partial}{\partial t_j},
\]
and the differential operator for the dilaton equation is
\[
L_0\,=\,-\frac32\frac{\partial\quad}{\partial t_1}\,+\,\sum_{j=0}^{\infty}\frac{2j+1}{2}t_j\frac{\partial}{\partial t_j}\,+\,\frac{1}{16}.
\]
With simple conditions, there is a unique way to extend operator definitions to obtain a representation of the $\{\mathcal L_n\}$ subalgebra.  The general operator is
\begin{multline*}
L_n\,=\,-\frac{(2n+3)!!}{2^{n+1}}\frac{\partial\quad }{\partial t_{n+1}}\,+\,\sum_{j=0}^{\infty}\frac{(2j+2n+1)!!}{(2j-1)!!\,2^{n+1}}\,t_j\, \frac{\partial\quad }{\partial t_{j+n}} \\
+\,\frac{\lambda^2}{2}\,\sum_{j=0}^{n-1}\frac{(2j+1)!!(2n-2j-1)!!}{2^{n+1}}\,\frac{\partial^2\qquad\quad }{\partial t_j\partial t_{n-j-1}},
\end{multline*}
with commutator $[ L_n, L_m]\,=\,(n-m) L_{n+m}$.  
\begin{thrm*}\cite[Theorem 6.1]{Mirwitt}. \textup{The Witten-Kontsevich conjecture: Virasoro constraints.}  For $n\ge -1$, then $L_n(e^{\mathbf F})\,=\,0.$  
\end{thrm*}
\begin{proof} 
For an exponents multi index $\mathbf k=(k_1,\dots,k_n)$, the volume recursion formula becomes the coefficient relation
\[
(2k_1+1)V_{g,n}(L)[\mathbf k]\,=\,\mathcal A_{g,n}^{con}(L)[\mathbf k]\,+\,\mathcal A_{g,n}^{dcon}(L)[\mathbf k]\,+\,\mathcal B_{g,n}(L)[\mathbf k].
\]
The leading coefficient relation takes the following explicit form 
(following the labeling of boundaries, the punctures are labeled $1,\dots,n$)
\begin{multline*}
(2k_1+1)!!\langle\tau_{k_1}\cdots\tau_{k_n}\rangle\\
=\frac12\,\sum_{i+j=k_1-2}(2i+1)!!(2j+1)!!\sum_{I\subset\{2,\dots,n\}}
\langle\tau_i\tau_{\mathbf k_{I}}\rangle\,\langle\tau_j\tau_{\mathbf k_{I^c}}\rangle\\
+\,\frac12\sum_{i+j=k_1-2}(2i+1)!!(2j+1)!!\langle\tau_i\tau_j\tau_{k_2}\cdots\tau_{k_n}\rangle\\
+\,\sum_{j=2}^n\frac{(2k_1+2k_j-1)!!}{(2k_j-1)!!}\langle\tau_{k_2}\cdots\tau_{k_1+k_j-1}\cdots\tau_{k_n}\rangle.
\end{multline*}
Rearranging the explicit relation provides that $L_{k_1-1}(e^{\mathbf F})=0$.  
\end{proof}

Mulase and Safnuk consider a generating function for the intersections of combinations of the $\kappa_1$ and $\psi$ classes \cite{SafMu}
\[
\mathbf G(s,t_0,t_1,\dots)\,=\,\sum_g\,\langle e^{s\kappa_1+\sum t_j\tau_j}\rangle_g\,=\,\sum_g\sum_{m,\{d_j \}}\langle\kappa_1^m\tau_0^{d_0}\tau_1^{d_1}\cdots\rangle_g\frac{s^m}{m!}
\prod_{j=0}^{\infty}\frac{t_j^{d_j}}{d_j!},
\]
where again products, other than $3g-3+n$-products, are defined as zero.  Mulase and Safnuk use the volume recursion and rearrangement of terms to prove the following.
\begin{thrm*}\cite[Thrm. 1.1]{SafMu} \textup{Virasoro constraints.}  
For each $k\ge -1$, define
\begin{multline*}
\mathcal V_k=-\frac{1}{2}\sum_{i=0}^{\infty}(2(i+k)+3)!!\frac{(-2s)^i}{(2i+1)!}\frac{\partial\qquad }{\partial t_{i+k+1}}+\frac12\sum_{j=0}^{\infty}\frac{(2(j+k)+1)!!}{(2j-1)!!}t_j\frac{\partial\ }{\partial t_{j+k}}\\
+\frac14\sum_{\stackrel{d_1+d_2=k-1}{d_1,d_2\ge 0}}(2d_1+1)!!(2d_2+1)!!\frac{\partial^2\quad }{\partial t_{d_1}\partial t_{d_2}}+\frac{\delta_{k,-1}t_0^2}{4}+\frac{\delta_{k,0}}{48},
\end{multline*}
for the double factorial and Kronecker delta function $\delta_{*,*}$.  
Then
\begin{itemize}
  \item the operators $\mathcal V_k$ satisfy the Virasoro commutator relations $[\mathcal V_n,\mathcal V_m]=(n-m)\mathcal V_{n+m}$; 
  \item the generating function $\mathbf G$ satisfies $\mathcal V_k(e^{\mathbf G})=0$ for $k\ge -1$.
\end{itemize}
The initial conditions and second system of equations uniquely determine the generating function. 
\end{thrm*}
In a direct display that the intersection numbers for $\kappa_1$ and $\psi$ classes are equivalent intersection numbers for $\psi$ classes, Mulase and Safnuk show that
\[
\mathbf G(s,t_0,t_1,t_2,t_3\dots)\,=\,\mathbf F(t_0,t_1,t_2+\gamma_2,t_3+\gamma_3,\dots),
\]
where $\gamma_j=-(-s)^{j-1}/(2j+1)j!$ \cite[Thrm. 1.2]{SafMu}.  An explicit proof of the relation also comes from a formula of Faber, expressing kappa classes in terms of psi classes on moduli spaces for a greater number of punctures.
\vspace{.1in}

In his thesis \cite{Dothe}, Norman Do presents a 
\[
\mbox{generalized string equation}\qquad V_{g,n+1}(L,2\pi i)\,=\,\sum_{k=1}^n\int L_kV_{g,n}(L)\,dL_k,
\]
and 
\[  
\mbox{generalized dilaton equation}\qquad\frac{\partial V_{g,n+1}}{\partial L_{n+1}}(L,2\pi i)\,=\,2\pi i (2g-2+n)V_{g,n}(L),
\]
where on the left hand side, the value $2\pi i$ is substituted for the $(n+1)^{st}$ boundary length and $L=(L_1,\dots,L_n)$.  By Theorem \ref{first}, the second equation, for appropriate non negative multi indices $\alpha=(\alpha_1,\alpha_2,\dots,\alpha_n)$ and integers $m$, is equivalent to the relations
\[
\int_{\overline{\mathcal M}_{g,n+1}}\psi_1^{\alpha_1}\psi_2^{\alpha_2}\cdots\psi_n^{\alpha_n}\psi_{n+1}(\kappa_1-\psi_{n+1})^m=(2g-2+n)\int_{\overline{\mathcal M}_{g,n}}\psi_1^{\alpha_1}\psi_2^{\alpha_2}\cdots\psi_n^{\alpha_n}\kappa_1^m.
\]
A proof of the generalized equations is based on the pullback relations for psi and kappa classes, and general considerations for images of classes.  In particular for 
$\pi:\overline{\mathcal M}_{g,n+1}\longrightarrow \overline{\mathcal M}_{g,n}$, the morphism of forgetting the last puncture, then the classes $\widetilde\kappa_m,\ \widetilde\psi_k$ on $\overline{\mathcal M}_{g,n+1}$ and $\kappa_m,\ \psi_k$ on $\overline{\mathcal M}_{g,n}$, satisfy $\widetilde\kappa_m=\pi^*\kappa_m+\psi_{n+1}^m$ and  $\widetilde\psi_k=\pi^*\psi_k,+D_k$ (for $D_k$ the divisor of the $k^{th}$ puncture on the universal curve over $\overline{\mathcal M}_{g,n}$), \cite{HMbook}.  

A second proof of the equations is based on exact formulas for the operations in the volume recursion.  For the generalized dilaton equation, consider the following four operators acting on the ring $\mathbb C[x^2,y^2,L_1^2,\dots,L_n^2,\dots]$,
\begin{multline*}
2\frac{\partial\ }{\partial L_1}L_1[\cdot],\qquad \frac{\partial\quad}{\partial L_{n+1}}[\cdot]\bigg |_{L_{n+1}=2\pi i},\qquad \int_0^{\infty}\int_0^{\infty}xyH(x+y,L_1)[\cdot]dxdy,\\ \mbox{and}\quad \int_0^{\infty} x\big(H(x,L_1+L_k)+H(x,L_1-L_k)\big)[\cdot]dx.\qquad   
\end{multline*}
Formulas for the operators are developed.  For the proof, the generalized dilaton equation is written using the second operator, the volume recursion is applied for the left hand side, operator formulas are applied, and terms are gathered to give the right hand side.  
The generalized dilaton equation gives WP volumes for the compact case, including the following examples
\begin{align*}
&V_{2,0}=\frac{43\pi^6}{2160},  &V_{3,0}&=\frac{176557\pi^{12}}{1209600},\\
&V_{4,0}=\frac{1959225867017\pi^{18}}{493807104000}\qquad \mbox{and}
\  &V_{5,0}&=\frac{84374265930915479\pi^{24}}{355541114880000}. 
\end{align*}

We close the lecture by noting that there is extensive research on Witten's conjecture for the moduli space.  The Kontsevich and Mirzakhani approaches are analytic in nature.  Okounkov and Pandharipande \cite{OP} transformed the question to counting Hurwitz numbers, topological types of branched covers of the sphere, and used a combinatorial approach to count factorizations of permutations into transpositions. Their combinatorial approach gives Kontsevich's formula.  There is an even greater body of research on intersection numbers and relations on the moduli space \cite{HMbook}.


\providecommand\WlName[1]{#1}\providecommand\WpName[1]{#1}\providecommand\Wl{W%
lf}\providecommand\Wp{Wlp}\def\cprime{$'$}

\end{document}